\documentclass{amsart}
\usepackage{amssymb}
\usepackage{amsmath}
\usepackage{url}
\usepackage{hyperref}
\usepackage{accents}
\usepackage{amscd}
\usepackage{amsmath}
\usepackage{scalefnt}
\usepackage{geometry}
\usepackage{color}

 \hypersetup{colorlinks=true}
\newtheorem{theorem}{Theorem}[section]
\newtheorem*{theoremA}{Theorem A}
\newtheorem{lemma}[theorem]{Lemma}

\numberwithin{equation}{section}
\newtheorem*{acknowledgement}{Acknowledgement}

\newtheorem{corollary}[theorem]{Corollary}

\newtheorem{proposition}[theorem]{Proposition}

\newtheorem{definition}[theorem]{Definition}
\theoremstyle{condition}

\theoremstyle{remark}
\newtheorem{remark}[theorem]{\bf{Remark}}
\newtheorem{notation}[theorem]{\bf{Notation}}
\newtheorem*{claimA}{\bf{Claim}}
\numberwithin{equation}{section}

\input{tcilatex}
\hypersetup{
    colorlinks,    citecolor=blue,    filecolor=black,    linkcolor=blue,    urlcolor=blue
}

\begin{document}
\title[Reductions of $2$-dimensional crystalline representations]{A note on
reductions of $2$-dimensional crystalline \\
Galois representations}
\author{Makis Dousmanis}
\subjclass[2000]{Primary 11F80, 11F85}
\maketitle

\begin{abstract}
Let $p$ be an odd prime number, $K_{f}$ the finite unramified extension of $%
%TCIMACRO{\U{211a} }%
%BeginExpansion
\mathbb{Q}
%EndExpansion
_{p}$ of degree $f,$ and $G_{K_{f}}$ its absolute Galois group. We construct
analytic families of \'{e}tale $\left( \varphi ,\Gamma \right) $-modules
which give rise to some families of $2$-dimensional crystalline
representations of $G_{K_{f}}$ with length of filtration $\geq p.$ As an
application, we prove that the modulo $p$ reductions of the members of each
such family (with respect to appropriately chosen Galois-stable lattices)
are constant.
\end{abstract}

\section*{Introduction}

Two-dimensional crystalline representations of $G_{%
%TCIMACRO{\U{211a} }%
%BeginExpansion
\mathbb{Q}
%EndExpansion
_{p}}:=\mathrm{Gal}\left( \overline{%
%TCIMACRO{\U{211a} }%
%BeginExpansion
\mathbb{Q}
%EndExpansion
_{p}}/%
%TCIMACRO{\U{211a} }%
%BeginExpansion
\mathbb{Q}
%EndExpansion
_{p}\right) $ arise from classical newforms of level prime to $p.$ Modulo $p$
reductions of such representations with Hodge-Tate weights in the range $%
[0;\ p]$ mattered for the proof of Serre's modularity conjecture by Khare
and Wintenberger. The reductions of all irreducible $2$-dimensional
crystalline representations of $G_{%
%TCIMACRO{\U{211a} }%
%BeginExpansion
\mathbb{Q}
%EndExpansion
_{p}}$ with Hodge-Tate weights in this range were computed by Berger-Li-Zhu 
\cite{BLZ04}, extending previous results of Deligne, Edixhoven, Fontaine and
Serre.

Serre's modularity conjecture has been recently generalized by Buzzard,
Diamond and Jarvis \cite{BDJ}, for irreducible totally odd $2$-dimensional $%
\overline{\mathbb{F}_{p}}$-representations of the absolute Galois group of
any totally real field unramified at $p.$ Two-dimensional crystalline
representations of $G_{K_{f}}:=\mathrm{Gal}\left( \overline{%
%TCIMACRO{\U{211a} }%
%BeginExpansion
\mathbb{Q}
%EndExpansion
_{p}}/K_{f}\right) $ arise naturally in the context of the BDJ conjecture,
and their modulo $p$ reductions are important for the formulation of the
weight part of this conjecture (see \cite[\S 3]{BDJ}). The weight part of
Serre's conjecture for rank two unitary groups, for modulo $p$
representations in the unramified case, has recently been proved by Gee, Liu
and Savitt \cite{GLS12}, using Liu's theory of $\left( \varphi ,\hat{G}%
\right) $-modules developed in \cite{LI10}. However, they have been unable
to explicitly determine the modulo $p$ reduction of a given $2$-dimensional
crystalline $\overline{%
%TCIMACRO{\U{211a} }%
%BeginExpansion
\mathbb{Q}
%EndExpansion
_{p}}$-representations of $G_{K_{f}}.$ Instead, they determined the possible
modulo $p$ reductions in terms of the Hodge-Tate weights, when the length of
the filtration of the corresponding admissible filtered $\varphi $-module is 
$\leq p.$

For arbitrary Hodge-Tate weights, semisimplified modulo $p$ reductions of
certain families of $2$-dimensional crystalline representations of $%
G_{K_{f}} $ were computed in \cite{DO10b}, extending the constructions in 
\cite{BLZ04} from $%
%TCIMACRO{\U{211a} }%
%BeginExpansion
\mathbb{Q}
%EndExpansion
_{p}$ to $K_{f}.$ More precisely, for any $2$-dimensional crystalline
representation $V$ of $G_{K_{f}}$ with Hodge-Tate type $\mathrm{HT}%
_{V}\left( \tau \right) =\{0,-k_{i}\},$ where the $k_{i}$ are nonnegative
integers, which is up to unramified twist either irreducible and induced
from a crystalline character of $G_{K_{2f}}\ $or split-reducible and
non-ordinary, we constructed an infinite family $\mathcal{F}\left( V\right) $
of $2$-dimensional crystalline representations of $G_{K_{f}}$ with the
following properties:

\begin{enumerate}
\item $V\in \mathcal{F}\left( V\right) ;$

\item The members of $\mathcal{F}\left( V\right) $ have Hodge-Tate type $%
\mathrm{HT}_{V}\left( \tau \right) ;$

\item The members of $\mathcal{F}\left( V\right) $ have the same modulo $p$\
reductions with respect to appropriately chosen Galois-stable lattices.
\end{enumerate}

The crystalline representations of $\mathcal{F}\left( V\right) $ were
described in terms of their corresponding by the Colmez-Fontaine theorem
(see \cite[Th\'{e}or\`{e}m A]{CF00}) admissible filtered $\varphi $-modules.
For each family $\mathcal{F}\left( V\right) ,$ the semisimplification $%
\overline{\mathcal{F}\left( V\right) }^{ss}$ of the common reduction is
independent of choices of lattices and was explicitly computed in (\cite[%
Theorems 1,5 \& 1.7]{DO10b}). Recall that, if $V$ is reducible, $\mathcal{F}%
\left( V\right) $ can contain both irreducible and reducible representations
(see \cite{DO10b}, comments after Theorem 1.7).

The modulo $p$ reduction of a given $2$-dimensional crystalline $\overline{%
%TCIMACRO{\U{211a} }%
%BeginExpansion
\mathbb{Q}
%EndExpansion
_{p}}$-representation of $G_{K_{f}}$ is generally unknown, even if the
length of filtration is $\leq p,\ $when $f\geq 2.$ The goal of this paper is
to enlarge the families $\mathcal{F}\left( V\right) $ to families of $2$%
-dimensional crystalline $\overline{%
%TCIMACRO{\U{211a} }%
%BeginExpansion
\mathbb{Q}
%EndExpansion
_{p}}$-representations of the same Hodge-Tate type and constant modulo $p$
reductions with respect to appropriately chosen Galois-stable lattices,
under the assumption that the length of filtration of $V$ is at least $p\neq
2.$ In addition to their application to Serre's modularity conjecture the
results of \cite{BLZ04} provided evidence for the existence of the $p$-adic
Langlands correspondence for $%
%TCIMACRO{\U{211a} }%
%BeginExpansion
\mathbb{Q}
%EndExpansion
_{p},$ and we expect that our results will similarly allow one to test the $%
p $-adic Langlands correspondence for $K_{f}$ currently being developed. The
proof of our theorem rests on Wach module constructions, and makes use of
the constructions in \cite{DO10b} and an idea of Berger (\cite[\S 10.3]{BB04}%
).

% $\alpha \beta \rho \alpha  \omikron \varsigma $ $\tau \varepsilon

\section{Description of the families\label{dof}}

Throughout this paper $p$\ will be a fixed odd integer prime, $K_{f}=\mathbb{%
Q}_{p^{f}}$\ the finite unramified extension of $\mathbb{Q}_{p}$\ of degree $%
f,$\ and $E$\ a finite extension of $K_{f}$\ with ring of integers $\mathcal{%
O}_{E},$\ maximal ideal $\mathfrak{m}_{E},\ $and residue field $k_{E}.$\
When the degree of $K_{f}$\ plays no role we simply write $K.$\ We denote by 
$\sigma _{K}$\ the absolute Frobenius of $K;$\ we fix once and for all a
distinguished embedding $K\overset{\tau _{0}}{\hookrightarrow }E$\ and we
let $\tau _{j}=\tau _{0}\circ \sigma _{K}^{j}$\ for all $j=0,1,...,f-1.$\ We
fix the $f$-tuple of embeddings $\mid \tau \mid :=(\tau _{0},\tau
_{1},...,\tau _{f-1});$ we denote $E^{\mid \tau \mid
}:=\tprod\nolimits_{\tau :K\hookrightarrow E}E,\ $with the embeddings
ordered as above, and we let $e_{i}:=(0,...,0,1_{i},0,...,0)\in E^{\mid \tau
\mid }$ for $i=0,1,...,f-1.$ For the language of crystalline representations
see \cite{FO88}.

\begin{notation}
\label{NOT}For $i=0,1,...,f-1,$\textit{\ let }$k_{i}$ \textit{be fixed
nonnegative integers which we call weights. Assume that after ordering them\
and omitting possibly repeated weights we get }$w_{0}<w_{1}<...<w_{t-1},$%
\textit{\ where }$w_{0}$\textit{\ is the smallest weight, }$w_{1}$\textit{\
the second smallest weight,\ ..., and }$w_{t-1}$\textit{\ is the largest
weight\ for some }$1\leq t\leq f.$\textit{\ The largest weight }$w_{t-1}$%
\textit{\ will be usually denoted by }$k\ $and throughout the paper we
assume that $k\geq p.$\textit{\ For convenience we define }$w_{-1}=0.$%
\textit{\ Let }$I_{0}:=\{0,1,...,f-1\};$\textit{\ for }$j=1,2,...,t-1$%
\textit{\ we let }$I_{j}:=\{i\in I_{0}:k_{i}>w_{j-1}\},$\textit{\ and }$%
I_{t}=\varnothing .$\textit{\ For each subset }$J\subset I_{0}$\textit{\ we
write }$f_{J}:=\tsum_{i\in J}e_{_{i}}$\textit{\ and }$E^{\mid \tau _{J}\mid
}:=f_{J}\cdot E^{\mid \tau \mid }.$\textit{\ The sets }$E^{\mid \tau
_{I_{j}}\mid }$\textit{\ are obtained as follows: }$E^{\mid \tau
_{I_{0}}\mid }$\textit{\ is the Cartesian product }$E^{f}.\ $\textit{%
Starting with }$E^{\mid \tau _{I_{0}}\mid },$\textit{\ we obtain }$E^{\mid
\tau _{I_{1}}\mid }$\textit{\ by killing\ the coordinates where the smallest
weight occurs.\ We obtain }$E^{\mid \tau _{I_{2}}\mid }$\textit{\ by further
killing\ the coordinates where the second smallest weight }$w_{1}$\textit{\
occurs and so on.}
\end{notation}

We first recall the construction of the families $\mathcal{F}\left( V\right)
\ $in \cite{DO10b}. For $i=0,1,...,f-1,$ let$\ \chi _{i}$ be a crystalline $%
E $-character of $G_{K_{f}}\ $with Hodge-Tate type $\mathrm{HT}_{\chi
_{i}}\left( \tau _{i+1}\right) =\{-1\}\ $and $\mathrm{HT}_{\chi _{j}}\left(
\tau _{i+1}\right) =\{0\}\ $if $j\neq i+1,\ $where the indices are viewed
modulo$\ f.$ Let $\{\ell _{j}\}_{0\leq j\leq 2f-1}$ be integers such that $%
\{\ell _{i},\ell _{f+i}\}=\{0,k_{i}\}$ for all $i=0,1,...,f-1.$ Up to
unramified twist, any irreducible $2$-dimensional crystalline representation 
$V$ of $G_{K_{f}}$ of Hodge-Tate type $\mathrm{HT}_{V}\left( \tau \right)
=\{0,-k_{i}\}$ induced from a crystalline character of $G_{K_{2f}}$ has the
form $V=\mathrm{Ind}_{K_{2f}}^{K_{f}}\left( \chi _{\vec{\ell}}\right) ,$
where $\chi _{\vec{\ell}}=\chi _{0}^{\ell _{1}}\cdot \chi _{1}^{\ell
_{2}}\cdots \chi _{2f-2}^{\ell _{2f-1}}\cdot \chi _{2f-1}^{\ell _{0}}\ $(cf. 
\cite[Theorem 1.3]{DO10b}). Any split-reducible non-ordinary $2$-dimensional
crystalline representation $V$ of $G_{K_{f}}$ of Hodge-Tate type $\mathrm{HT}%
_{V}\left( \tau \right) =\{0,-k_{i}\}$ is up to unramified twist of the form%
\begin{equation*}
V=\eta \cdot \chi _{0}^{\ell _{1}}\cdot \chi _{1}^{\ell _{2}}\cdots \chi
_{f-2}^{\ell _{f-1}}\cdot \chi _{f-1}^{\ell _{0}}\tbigoplus \chi _{0}^{\ell
_{1+f}}\cdot \chi _{1}^{\ell _{2+f}}\cdots \chi _{f-2}^{\ell _{2f-1}}\cdot
\chi _{f-1}^{\ell _{f}},
\end{equation*}%
where $\eta $ is an unramified character, with both vectors $\left( \ell
_{0},\ell _{1},...,\ell _{f-1}\right) $ and $\left( \ell _{f},\ell
_{f+1},...,\ell _{2f-1}\right) $ nonzero (cf. \cite[Theorem 1.7]{DO10b}).
Fix a representation $V\ $as above. Let $\{X_{i}\}_{1\leq i\leq f}$\ be a
set of indeterminates and let $P_{i}\left( X_{i}\right) \in M_{2}\left( 
\mathcal{O}_{E}[X_{i}]\right) $ be a matrix of one of the following four
types:\ 
\begin{equation*}
t_{1}\mathbf{:}\left( 
\begin{array}{cc}
p^{k_{i}} & 0 \\ 
X_{i} & 1%
\end{array}%
\right) ,\ \ \ \ t_{2}\mathbf{:}\left( 
\begin{array}{cc}
X_{i} & 1 \\ 
p^{k_{i}} & 0%
\end{array}%
\right) ,\ \ \ \ t_{3}\mathbf{:}\left( 
\begin{array}{cc}
1 & X_{i} \\ 
0 & p^{k_{i}}%
\end{array}%
\right) ,\ \ \ \ t_{4}:\left( 
\begin{array}{cc}
0 & p^{k_{i}} \\ 
1 & X_{i}%
\end{array}%
\right) .
\end{equation*}%
Let $P(\overrightarrow{X})=\left( P_{1}\left( X_{1}\right) ,P_{2}\left(
X_{2}\right) ,...,P_{f}\left( X_{f}\right) \right) ,$ where the indices are
viewed $\func{mod}$ $f,$ and choose the type of the matrix $P_{i}\left(
X_{i}\right) $ as follows: %$\epsilon \kappa \alpha \tau \eta, 
\noindent If $f=1$ we choose $P_{1}=t_{2}.$ A{}ssume that $f\geq 2.$

\noindent Case (i). $V$ is induced.

(1) If $\ell _{1}=0,$ $P_{1}=t_{2};$

(\noindent 2) If $\ell _{1}=k_{1}>0,$ $P_{1}=t_{1}.$ \noindent

\noindent For $i=2,3,...,f-1$ we choose the type of the matrix $P_{i}$ as
follows:\noindent

(1) If $\ell _{i}=0,$ then: \noindent

\begin{itemize}
\item If an even number of coordinates of $(P_{1},P_{2},...,P_{i-1})$ is of
even type, $P_{i}=t_{2};$ \noindent 
%\sigma \alpha \xi $ $\beta \omikron \eta \theta

\item If an odd number of coordinates of $(P_{1},P_{2},...,P_{i-1})$ is of
even type, $P_{i}=t_{1}.$
\end{itemize}

\noindent

(\noindent 2) If $\ell _{i}=k_{i}>0,$ then: \noindent

\begin{itemize}
\item If an even number of coordinates of $(P_{1},P_{2},...,P_{i-1})$ is of
even type, $P_{i}=t_{1};$ \noindent

\item If an odd number of coordinates of $(P_{1},P_{2},...,P_{i-1})$ is of
even type, $P_{i}=t_{2}.$
\end{itemize}

%\Gamma \alpha \mu \omega \tau \omikron \nu 
\noindent Finally, we choose the type of the matrix $P_{0}:=P_{f}$ as
follows:

\noindent

(\noindent 1) If $\ell _{0}=0,$ then:

\begin{itemize}
\item \noindent If an even number of coordinates of $%
(P_{1},P_{2},...,P_{f-1})$ is of even type, $P_{0}=t_{4};$

\item If an odd number of coordinates of $(P_{1},P_{2},...,P_{f-1})$ is of
even type, $P_{0}=t_{3}.$
\end{itemize}

(\noindent 2) If $\ell _{0}=k_{0}>0,$ then:

\begin{itemize}
\item If an even number of coordinates of $(P_{1},P_{2},...,P_{f-1})$ is of
even type, $P_{0}=t_{2};$

\item If an odd number of coordinates of $(P_{1},P_{2},...,P_{f-1})$ is of
even type, $P_{0}=t_{1}.$
\end{itemize}

\noindent Case (ii). $V$ is split reducible and non-ordinary.

\noindent The $\left( f-1\right) $-tuple $\left(
P_{1},P_{2},...,P_{f-1}\right) $ is chosen as in Case (i) above. If $\eta
=\eta _{c}$ is the unramified character which maps the geometric Frobenius
element $\mathrm{Frob}_{K_{f}}$ of $G_{K_{f}}$ to $c,$ we replace the entry $%
p^{k_{0}}$ in the definition of the matrix $P_{0}$ by $cp^{k_{0}}.$ The type
of the matrix $P_{0}:=P_{f}$ is chosen as follows: \noindent

(1) If $\ell _{0}=0,$ then: \noindent

\begin{itemize}
\item If an even number of coordinates of $\left(
P_{1},P_{2},...,P_{f-1}\right) $ is of even type, $P_{0}=t_{3};$

\item If an odd number of coordinates of $\left(
P_{1},P_{2},...,P_{f-1}\right) $ is of even type, $P_{0}=t_{4}.$
\end{itemize}

\noindent

(2) If $\ell _{0}=k_{0}>0,$ then:

\begin{itemize}
\item If an even number of coordinates of $\left(
P_{1},P_{2},...,P_{f-1}\right) $ is of even type, $P_{0}=t_{1};$

\item If an odd number of coordinates of $\left(
P_{1},P_{2},...,P_{f-1}\right) $ is of even type, $P_{0}=t_{2}.$
\end{itemize}

\noindent Recall that $k\geq p$ and let%
\begin{equation}
m:=\left\{ 
\begin{array}{l}
\lfloor \frac{k-1}{p-1}\rfloor \ \ \ \ \ \mathrm{if}\mathnormal{\ \ }%
k_{i}\neq p\ \text{for some }i, \\ 
\ \ \ 0\ \ \text{\ }\ \ \ \ \ \mathrm{if}\mathnormal{\ }\text{ }k_{i}=p\ 
\text{for all }i.%
\end{array}%
\right.  \label{m}
\end{equation}%
For any $\vec{\alpha}=\left( \alpha _{1},\alpha _{2},...,\alpha _{f}\right)
\in \left( p^{m}\mathfrak{m}_{E}\right) ^{f},$ let $P\left( \vec{\alpha}%
\right) $ be the matrix obtained by evaluating each indeterminate $X_{i}$ at 
$\alpha _{i}.$ We defined $\mathcal{F}\left( V\right) $ as the family of $2$%
-dimensional crystalline representations $\{V\left( \vec{\alpha}\right) ,\ 
\vec{\alpha}\in \left( p^{m}\mathfrak{m}_{E}\right) ^{f}\}\ $corresponding
by the Colmez-Fontaine theorem to the family of admissible filtered $\varphi 
$-modules obtained by equipping $\mathbb{D}(\vec{\alpha})=E^{\mid \tau \mid
}\eta _{1}\tbigoplus E^{\mid \tau \mid }\eta _{2}$ with the Frobenius action
defined by $\left( \varphi \left( \eta _{1}\right) ,\varphi \left( \eta
_{2}\right) \right) =\left( \eta _{1},\eta _{2}\right) P(\vec{\alpha})$ and
the filtration 
\begin{equation*}
\ \ \ \ \ \ \ \ \ \ \ \ \mathrm{Fil}^{\mathrm{j}}(\mathbb{D}\left( \vec{%
\alpha}\right) )=\left\{ 
\begin{array}{l}
E^{\mid \tau \mid }\eta _{1}\bigoplus E^{\mid \tau \mid }\eta _{2}\ \ \ \ \
\ \text{if\ \ \ }j\leq 0, \\ 
E^{\mid \tau _{I_{s}}\mid }\left( \vec{x}\eta _{1}+\vec{y}\eta _{2}\right) \
\ \ \text{if \ }\ 1+w_{s-1}\leq j\leq w_{s},\ \text{for }s=0,1,...,t-1, \\ 
\ \ \ \ \ \ \ \ \ \ 0\ \ \ \ \ \ \ \ \ \ \ \ \ \ \ \mathrm{if\ \ }\text{\ }%
j\geq 1+w_{t-1},%
\end{array}%
\right.
\end{equation*}%
where $\vec{x}=(x_{0},x_{1},...,x_{f-1})$ and $\vec{y}%
=(y_{0},y_{1},...,y_{f-1}),$ with 
\begin{equation*}
(x_{i},y_{i})=\left\{ 
\begin{array}{l}
(1,-\alpha _{i})\text{\ \ \ \ \ if }P_{i}\ \text{has type }1\ \text{or}\ 2,
\\ 
(-\alpha _{i},1)\ \ \ \ \ \text{if }P_{i}\ \text{has type }3\ \text{or }4,%
\end{array}%
\right.
\end{equation*}%
for any $\vec{\alpha}\in \left( p^{m}\mathfrak{m}_{E}\right) ^{f}.$ By the
construction of these families in \cite{DO10b} it follows that $V(\vec{0}%
)=V. $ We now enlarge each such family $\mathcal{F}\left( V\right) ,\ $%
preserving the Hodge-Tate types, and leaving unchanged the modulo$\ p$
reductions with respect to appropriately chosen Galois-stable $\mathcal{O}%
_{E}$-lattices.

Let $\alpha \left( k\right) :=\tsum\nolimits_{n=0}^{\infty }\lfloor \frac{k}{%
p^{n}\left( p-1\right) }\rfloor .$ For any $A=\left(
A_{1},A,...,A_{f}\right) \in M_{2}\left( p^{\alpha \left( k-1\right) }%
\mathcal{O}_{E}\right) ^{\mid \tau \mid }$ we define%
\begin{equation*}
P_{A}(\overrightarrow{X}):=\left( \overrightarrow{Id}+A\right) P\left( 
\overrightarrow{X}\right) .
\end{equation*}%
If $\vec{\alpha}\in \left( p^{m}\mathfrak{m}_{E}\right) ^{f}$ and $A\in
M_{2}\left( p^{\alpha \left( k-1\right) }\mathcal{O}_{E}\right) ^{\mid \tau
\mid },$ we denote by $\left( \mathbb{D}_{A\ }(\vec{\alpha}),\varphi \right) 
$ the filtered $\varphi $-module obtained by equipping $\mathbb{D}_{A}(\vec{%
\alpha})=E^{\mid \tau \mid }\eta _{1}\tbigoplus E^{\mid \tau \mid }\eta _{2}$
with the Frobenius endomorphism defined by $\left( \varphi \left( \eta
_{1}\right) ,\varphi \left( \eta _{2}\right) \right) =\left( \eta _{1},\eta
_{2}\right) P_{A}(\vec{\alpha})$ and with the same filtration as $\left( 
\mathbb{D}(\vec{\alpha}),\varphi \right) $ independently of $A.$ Such a
filtered $\varphi $-module turns out to be admissible. Let $V_{A}\left( \vec{%
\alpha}\right) $ be the crystalline representation corresponding by the
Colmez-Fontaine theorem to $\left( \mathbb{D}_{A\ }(\vec{\alpha}),\varphi
\right) ,$ and let%
\begin{equation*}
\mathcal{G}\left( V\right) =\tbigcup\nolimits_{A\in \mathcal{M}\left(
k\right) }\left\{ V_{A}\left( \vec{\alpha}\right) :\ \vec{\alpha}\in \left(
p^{m}\mathfrak{m}_{E}\right) ^{f}\right\} ,\ \text{where}\ \mathcal{M}\left(
k\right) =M_{2}\left( p^{1+\alpha \left( k-1\right) }\mathcal{O}_{E}\right)
^{\mid \tau \mid }.
\end{equation*}

\begin{theoremA}
\label{thm}

\begin{enumerate}
\item[(i)] For any $\vec{\alpha}\in \left( p^{m}\mathfrak{m}_{E}\right) ^{f}$
and any $A\in M_{2}\left( p^{\alpha \left( k-1\right) }\mathcal{O}%
_{E}\right) ^{\mid \tau \mid },\ $the filtered $\varphi $-modules $\mathbb{D}%
_{A\ }(\vec{\alpha})$ are admissible and the corresponding crystalline
representations have Hodge-Tate type $\mathrm{HT}\left( \tau _{i}\right)
=\{0,-k_{i}\};$

\item[(ii)] For any $\vec{\alpha}\in \left( p^{m}\mathfrak{m}_{E}\right)
^{f} $ and any $A\in M_{2}\left( p^{\alpha \left( k-1\right) }\mathcal{O}%
_{E}\right) ^{\mid \tau \mid },$ there exist $G_{K_{f}}$-stable $\mathcal{O}%
_{E}$-lattices with respect to which $\overline{V}_{A}\left( \vec{\alpha}%
\right) =\overline{V}_{A}(\vec{0});$ 
% $ \varepsilon \beta \rho \alpha \iota \omikron$

\item[(iii)] There exist $G_{K_{f}}$-stable $\mathcal{O}_{E}$-lattices with
respect to which all members of $\mathcal{G}(V)$ have the same modulo$\ p$
reduction $\overline{\mathcal{G}\left( V\right) }.$ Moreover, $\overline{%
\mathcal{G}\left( V\right) }=\overline{V}.$
\end{enumerate}
\end{theoremA}

\begin{remark}

\begin{enumerate}
\item By (\cite[Theorems 1.5 \& 1.7]{DO10b}), 
\begin{equation*}
\left( \overline{\mathcal{G}\left( V\right) }_{\mid I_{K_{f}}}\right)
^{ss}=\left\{ 
\begin{array}{l}
\omega _{2f,\bar{\tau}_{0}}^{\beta }\oplus \omega _{2f,\bar{\tau}%
_{0}}^{p^{f}\beta },\ \text{where\ }\beta =-\tsum\nolimits_{i=0}^{2f-1}\ell
_{i}p^{i}\ \text{if }V\ \text{is irreducible and induced,} \\ 
\\ 
\omega _{f,\bar{\tau}_{0}}^{\beta }\oplus \omega _{f,\bar{\tau}_{0}}^{\beta
^{\prime }},\ \text{where\ }\beta =-\tsum\nolimits_{i=0}^{f-1}\ell
_{i}p^{i}\ \text{and}\ \beta ^{\prime }=-\tsum\nolimits_{i=0}^{f-1}\ell
_{i+f}p^{i}\ \text{if }V\ \text{is } \\ 
\\ 
\text{split-reducible and nonordinary.}%
\end{array}%
\right.
\end{equation*}%
Recall that the level $f$ fundamental character $\omega _{f,\bar{\tau}%
_{0}}:I_{K_{f}}\overset{}{\rightarrow }k_{E}^{\times }$ is obtained by
composing the homomorphism $I_{K_{f}}\rightarrow k_{K_{f}}^{\times }$
obtained from local class field theory (with uniformizers corresponding to
geometric Frobenius elements) with the embedding of residue fields $k_{K_{f}}%
\overset{\bar{\tau}_{0}}{\rightarrow }k_{E}$ obtained from the distinguished
embedding$\ K\overset{\tau _{0}}{\hookrightarrow }E.$

\item For the rest of this remark assume that $f\geq 2.$ In this case, for a 
$2$-dimensional crystalline $\overline{%
%TCIMACRO{\U{211a} }%
%BeginExpansion
\mathbb{Q}
%EndExpansion
_{p}}$-representation $V$ of $G_{K_{f}},$ the characteristic polynomial of
Frobenius and a choice of the filtration of the corresponding by the
Colmez-Fontaine theorem admissible filtered $\varphi $-module $\mathbb{D}%
\left( V\right) $ fail to determine its isomorphism class. Assuming that $%
\mathbb{D}\left( V\right) $ is Frobenius-semisimple and non-Frobenius-scalar
and fixing the characteristic polynomial of Frobenius and a choice for the
filtration, the additional datum required to determine the isomorphism class
of $V$ is (roughly) an element of $\mathbb{P}^{f-1}\left( E\right) \ $(for a
precise statement see \cite[\S 7]{DO10a}). The isomorphism classes of
non-Frobenius-semisimple or Frobenius-scalar filtered $\varphi $-modules are
in general messier to describe (see \cite[\S 6]{DO10a}).

\item The representations of $\mathcal{G}\left( V\right) $ yield additional
\textquotedblleft projective parameters\textquotedblright\ compared to the
set of \textquotedblleft projective parameters\textquotedblright\ attached
to the Frobenius-semisimple and non-Frobenius-scalar unramified twists of
members of $\mathcal{F}\left( V\right) .$ However, they yield no new
characteristic polynomials or filtrations.

\item The formulas for the \textquotedblleft projective
parameters\textquotedblright\ of the Frobenius-semisimple and
non-Frobenius-scalar representations of the families $\mathcal{G}\left(
V\right) $ look particularly abhorrent (see for instance the proof of \cite[%
Proposition 6.21]{DO10b}). The situation becomes even worse with the
non-Frobenius-semisimple, and in especially with the Frobenius-scalar
members of these families. This makes it hard to give a clean description in
terms of the classification of admissible filtered $\varphi $-modules
obtained in \cite{DO10a} of how many $2$-dimensional crystalline
representations of $G_{K_{f}}$ with Hodge-Tate weights in the range $[0;\ p]$
we are able to compute the semisimplified modulo $p$ reduction of, using
Theorem A, and what is possibly missing.

\item Theorem A can be thought of as a local constancy result for modulo $p$
reductions of $2$-dimensional crystalline representations of $G_{K_{f}}\ $%
within certain families. For results of similar flavor for $2$-dimensional
crystalline representations of $G_{%
%TCIMACRO{\U{211a} }%
%BeginExpansion
\mathbb{Q}
%EndExpansion
_{p}},$ see \cite{Ber12}.
\end{enumerate}
\end{remark}

\section{Families of Wach modules}

\subsection{\'{E}tale $\left(\protect\varphi ,\Gamma \right) $-modules and
Wach modules}

Let $\mathcal{K}_{n}=K(\zeta _{p^{n}}),\ $where $\zeta _{p^{n}}\ $is a
primitive $\ p^{n}$-th$\ $root\ of\ unity$\ $inside $\overline{%
%TCIMACRO{\U{211a} }%
%BeginExpansion
\mathbb{Q}
%EndExpansion
_{p},}\ $and let $K_{\infty }=\cup _{n\geq 1}\mathcal{K}_{n}.$ Let $\chi
:G_{K}\rightarrow 
%TCIMACRO{\U{2124} }%
%BeginExpansion
\mathbb{Z}
%EndExpansion
_{p}^{\times }$ be the cyclotomic character, $H_{K}:=\ker \chi =\mathrm{Gal}(%
\overline{%
%TCIMACRO{\U{211a} }%
%BeginExpansion
\mathbb{Q}
%EndExpansion
_{p}}/K_{\infty }),$\ and $\Gamma _{K}:=G_{K}/H_{K}=\mathrm{Gal}(K_{\infty
}/K).$ Fontaine \cite{FO90} has constructed topological rings $\mathbb{A}$
and $\mathbb{B}$ endowed with continuous commuting Frobenius $\varphi $ and $%
G_{%
%TCIMACRO{\U{211a} }%
%BeginExpansion
\mathbb{Q}
%EndExpansion
_{p}}$-actions. Unless otherwise stated and whenever applicable, continuity
will mean continuity with respect to the topologies induced by the weak
topologies of the topological rings $\mathbb{A}$ and $\mathbb{B}.$ Let $%
\mathbb{A}_{K}=\mathbb{A}^{H_{K}}$ and $\mathbb{B}_{K}=\mathbb{B}^{H_{K}},$
and let $\mathbb{A}_{K\,,E}:=\mathcal{O}_{E}\otimes _{%
%TCIMACRO{\U{2124} }%
%BeginExpansion
\mathbb{Z}
%EndExpansion
_{p}}\mathbb{A}_{K}$ and $\mathbb{B}_{K\,,E}:=E\otimes _{%
%TCIMACRO{\U{211a} }%
%BeginExpansion
\mathbb{Q}
%EndExpansion
_{p}}\mathbb{B}_{K}.$ The actions of $\varphi $ and $\Gamma _{K}$ extend to $%
\mathbb{A}_{K,E}\ $and $\mathbb{B}_{K\,,E}$ by $\mathcal{O}_{E}$ (resp. $E$%
)-linearity, and one easily sees that $\mathbb{A}_{K,E}=\mathbb{A}%
_{E}^{H_{K}}$ and $\mathbb{B}_{K,E}=\mathbb{B}_{E}^{H_{K}}.$

\begin{definition}
\label{Wach}A $(\varphi ,\Gamma )$-module over $\mathbb{A}_{K,E}$\ (resp. $%
\mathbb{B}_{K,E}$) is an $\mathbb{A}_{K,E}$-module of finite type (resp. a
free $\mathbb{B}_{K,E}$-module of finite type) endowed with a semilinear and
continuous action of $\Gamma _{K},\ $ and with a semilinear map $\varphi $
which commutes with the action of $\Gamma _{K}.\ $ A $(\varphi ,\Gamma )$%
-module $M$\ over $\mathbb{A}_{K,E}$\ is called \'{e}tale if $\varphi ^{\ast
}(M)=M,$\ where $\varphi ^{\ast }(M)$\ is the $\mathbb{A}_{K,E}$-module
generated by the set $\varphi (M).$\ A $(\varphi ,\Gamma )$-module $M$\ over 
$\mathbb{B}_{K,E}$\ is called \'{e}tale if it contains a basis $%
(e_{1},...,e_{d})$\ over $\mathbb{B}_{K,E}$\ such that $\left( \varphi
(e_{1}),...,\varphi (e_{d})\right) =(e_{1},...,e_{d})A,$\ for some matrix $%
A\in \mathrm{GL}_{d}\left( \mathbb{A}_{K,E}\right) .$
\end{definition}

\noindent \indent If $V$ is a continuous $E$-linear representation of $G_{K}$
we equip the $\mathbb{B}_{K,E}$-module $\mathbb{D}(V):=\left( \mathbb{B}%
_{E}\otimes _{E}V\right) ^{H_{K}}$ with a Frobenius endomorphism $\varphi $
defined by $\varphi (b\otimes v):=\varphi (b)\otimes v,$ where $\varphi $ on
the right hand side is the Frobenius of $\mathbb{B}_{E},$ and with an action
of $\Gamma _{K}$ given by$\ \bar{g}(b\otimes v):=gb\otimes gv$ for any $g\in
G_{K}.$ This $\Gamma _{K}$-action commutes with $\varphi $ and is
continuous. Moreover, $\mathbb{D}(V)$ is an \'{e}tale $\left( \varphi
,\Gamma \right) $-module over $\mathbb{B}_{K,E}.$ Conversely, if $D$ is an 
\'{e}tale $\left( \varphi ,\Gamma \right) $-module over $\mathbb{B}_{K,E},$
let $\mathbb{V}(D):=\left( \mathbb{B}_{E}\otimes _{\mathbb{B}_{K,E}}D\right)
^{\varphi =1},$ where $\varphi (b\otimes d):=\varphi (b)\otimes \varphi (d).$
The $E$-vector space $\mathbb{V}(D)$ is finite dimensional and is equipped
with a continuous $E$-linear $G_{K}$-action given by $g(b\otimes
d):=gb\otimes \bar{g}d.$ We have the following theorem of Fontaine.

\begin{theorem}
\cite{FO90}\label{fontaine|s equivalence}

\begin{enumerate}
\item[(i)] There is an equivalence of categories between continuous $E$%
-linear representations of $G_{K}\ $and \'{e}tale $(\varphi ,\Gamma )$%
-modules over $\mathbb{B}_{K,E}$ given by%
\begin{equation*}
\mathbb{D}:\mathrm{Rep}_{E}\left( G_{K}\right) \rightarrow \mathcal{M}%
od_{(\varphi ,\Gamma )}^{\ \acute{e}t}\left( \mathbb{B}_{K,E}\right)
:V\longmapsto \mathbb{D}(V):=\left( \mathbb{B}_{E}\otimes _{E}V\right)
^{H_{K}},
\end{equation*}%
with quasi-inverse functor 
\begin{equation*}
\mathbb{V}:\mathcal{M}od_{(\varphi ,\Gamma )}^{\ \acute{e}t}\left( \mathbb{B}%
_{K,E}\right) \rightarrow \mathrm{Rep}_{E}\left( G_{K}\right) :D\longmapsto 
\mathbb{V}(D):=\left( \mathbb{B}_{E}\otimes _{\mathbb{B}_{K,E}}D\right)
^{\varphi =1}.
\end{equation*}

\item[(ii)] There is an equivalence of categories between continuous $%
\mathcal{O}_{E}$-linear representations of $G_{K}\ $and \'{e}tale $(\varphi
,\Gamma )$-modules over $\mathbb{A}_{K,E}$ given by 
\begin{equation*}
\mathbb{D}:\mathrm{Rep}_{\mathcal{O}_{E}}\left( G_{K}\right) \rightarrow 
\mathcal{M}od_{(\varphi ,\Gamma )}^{\ \acute{e}t}\left( \mathbb{A}%
_{K,E}\right) :\mathrm{T}\longmapsto \mathbb{D}(T):=\left( \mathbb{A}%
_{E}\otimes _{\mathcal{O}_{E}}\mathrm{T}\right) ^{H_{K}},
\end{equation*}%
with quasi-inverse functor 
\begin{equation*}
\mathbb{T}:\mathcal{M}od_{(\varphi ,\Gamma )}^{\ \acute{e}t}\left( \mathbb{A}%
_{K,E}\right) \rightarrow \mathrm{Rep}_{\mathcal{O}_{E}}\left( G_{K}\right)
:D\longmapsto \mathbb{T}(D):=\left( \mathbb{A}_{E}\otimes _{\mathbb{A}%
_{K,E}}D\right) ^{\varphi =1}.\noindent
\end{equation*}
\end{enumerate}
\end{theorem}

Let $\mathbb{A}_{K}=\{\tsum_{-\infty }^{+\infty }\alpha _{n}\pi ^{n}:\alpha
_{n}\in \mathcal{O}_{K}\ \mathrm{and}\underset{n\rightarrow -\infty }{\lim }%
\alpha _{n}=0\}$ for some element $\pi \ $which can be thought of as a
formal variable. The ring $\mathbb{A}_{K}$ is equipped with a Frobenius
endomorphism $\varphi $ which extends the absolute Frobenius of $\mathcal{O}%
_{K}$ and is such that $\varphi (\pi )=(1+\pi )^{p}-1.$ It is also equipped
with a $\Gamma _{K}$-action which is $\mathcal{O}_{K}$-linear, commutes with
Frobenius, and is such that $\gamma (\pi )=(1+\pi )^{\chi (\gamma )}-1$ for
all $\gamma \in \Gamma _{K}.\ $The ring $\mathbb{A}_{K}$ is a local domain
with maximal ideal $(p)$ and fraction field $\mathbb{B}_{K}=\mathbb{A}_{K}[%
\frac{1}{p}].$ The rings $\mathbb{A}_{K},\ \mathbb{A}_{K,E},\ \mathbb{B}_{K}$
and $\mathbb{B}_{K,E}$ contain the subrings $\mathbb{A}_{K}^{+}=\mathcal{O}%
_{K}[[\pi ]]$\noindent $,\ \mathbb{A}_{K,E}^{+}:=$ $\mathcal{O}_{E}\otimes _{%
\mathbb{Z}_{p}}\mathbb{A}_{K}^{+},\ \mathbb{B}_{K}^{+}=$ $\mathbb{A}_{K}^{+}[%
\frac{1}{p}]\ $and $\mathbb{B}_{K,E}^{+}:=$ $E\otimes _{%
%TCIMACRO{\U{211a} }%
%BeginExpansion
\mathbb{Q}
%EndExpansion
_{p}}\mathbb{B}_{K}^{+}$ respectively, and these subrings are equipped with
the restrictions of the $\varphi $ and the $\Gamma _{K}$-actions of the
rings containing them. The map $\upsilon :\mathbb{A}_{K,E}^{+}\rightarrow
\tprod\nolimits_{\tau :K\hookrightarrow E}\mathcal{O}_{E}[[\pi ]]$ given by $%
\upsilon \left( a\otimes b\right) =\left( a\tau _{0}\left( b\right) ,a\tau
_{1}\left( b\right) ,...,a\tau _{f-1}\left( b\right) \right) ,$ where $\tau
_{i}\left( \tsum\nolimits_{n=0}^{\infty }\beta _{n}\pi ^{n}\right)
=\tsum\nolimits_{n=0}^{\infty }\tau _{i}\left( \beta _{n}\right) \pi ^{n}$
for all $b=\tsum\nolimits_{n=0}^{\infty }\beta _{n}\pi ^{n}\in \mathbb{A}%
_{K}^{+}$ is a ring isomorphism. The ring\ $\mathcal{O}_{E}[[\pi ]]^{\mid
\tau \mid }:=\tprod\nolimits_{\tau :K\hookrightarrow E}\mathcal{O}_{E}[[\pi
]]$ is equipped via $\upsilon $ with commuting $\mathcal{O}_{E}$-linear
actions of $\varphi $ and $\Gamma _{K}$ given by the formulas$\ \ \ \ \ \ \
\ \ \ \ \ \ $%
\begin{align}
& \varphi (\alpha _{0}(\pi ),\alpha _{1}(\pi ),...,\alpha _{f-1}(\pi
))=(\alpha _{1}(\varphi (\pi )),...,\alpha _{f-1}(\varphi (\pi )),\alpha
_{0}(\varphi (\pi )))\noindent \ \mathrm{and}  \label{actions1} \\
& \gamma (\alpha _{0}(\pi ),\alpha _{1}(\pi ),...,\alpha _{f-1}(\pi
))=(\alpha _{0}(\gamma \pi ),\alpha _{1}(\gamma \pi ),...,\alpha
_{f-1}(\gamma \pi ))  \label{actions2}
\end{align}%
for all $\gamma \in \Gamma _{K}.$

\begin{definition}
Suppose $k\geq 0.$\ A Wach module over $\mathbb{A}_{K,E}^{+}$\ (resp. $%
\mathbb{B}_{K,E}^{+}$) with weights in $[-k;\ 0]$\ is a free $\mathbb{A}%
_{K,E}^{+}$-module (resp. $\mathbb{B}_{K,E}^{+}$-module) $N$\ of finite
rank, endowed with an action of $\Gamma _{K}$\ which becomes trivial modulo $%
\pi $, and also with a Frobenius map $\varphi $\ which commutes with the
action of $\Gamma _{K}$\ and such that $\varphi (N)\subset N$\ and $%
N/\varphi ^{\ast }(N)$\ is killed by $q^{k},\ $where $q:=\varphi (\pi )/\pi $
and $\varphi ^{\ast }(N)$ is the $\mathbb{A}_{K,E}^{+}$-module (resp. $%
\mathbb{B}_{K,E}^{+}$-module) generated by the set $\varphi \left( N\right)
. $
\end{definition}

The following theorem of Berger determines which types of \'{e}tale $%
(\varphi ,\Gamma )$-modules correspond to crystalline representations via
Fontaine's functor.

\begin{theorem}
\cite{Ber03} \label{berger thm}

\begin{enumerate}
\item[(i)] An $E$-linear representation $V$ of $G_{K}$ is crystalline with
Hodge-Tate weights in $[-k;\ 0]$ if and only if $\mathbb{D}(V)$ contains a
unique Wach module $\mathbb{N}(V)$ of rank $\dim _{E}V$ with weights in $%
[-k;\ 0].$ The functor $V\mapsto \mathbb{N}(V)$ defines an equivalence of
categories between crystalline representations of $G_{K}$ and Wach modules
over $\mathbb{B}_{K,E}^{+},$ compatible with tensor products, duality and
exact sequences. \noindent

\item[(ii)] For a given crystalline $E$-representation $V,$ the map $\mathrm{%
T}\mapsto \mathbb{N}(\mathrm{T}):=\mathbb{N}(V)\cap \mathbb{D}(\mathrm{T})$
induces a bijection between $G_{K}$-stable, $\mathcal{O}_{E}$-lattices of $V$
and Wach modules over $\mathbb{A}_{K,E}^{+}$ which are $\mathbb{A}_{K,E}^{+}$%
-lattices contained in $\mathbb{N}(V).$ Moreover $\mathbb{D}(\mathrm{T})=%
\mathbb{A}_{K,E}\otimes _{\mathbb{A}_{K,E}^{+}}\mathbb{N}(\mathrm{T})$.

\item[(iii)] If $V$ is a crystalline $E$-representation of $G_{K},$ and if
we endow $\mathbb{N}(V)$ with the filtration $\mathrm{Fil}^{\mathrm{j}}%
\mathbb{N}(V)=\{x\in \mathbb{N}(V)|\varphi (x)\in q^{j}\mathbb{N}(V)\},$
then we have an isomorphism 
\begin{equation*}
\mathbb{D}_{\mathrm{cris}}(V)\rightarrow E^{\mid \tau \mid }\otimes _{%
\mathcal{O}_{E}}\mathbb{N}(V)/\pi \mathbb{N}(V)
\end{equation*}%
of filtered $\varphi $-modules over $E^{\mid \tau \mid }$ (with the induced
filtration on $\mathbb{N}(V)/\pi \mathbb{N}(V)$).
\end{enumerate}
\end{theorem}

\subsection{Construction of families of Wach modules\label{newfromold}}

We fix a topological generator $\delta $ of the procyclic group $\Gamma
_{K}. $ For any positive integer $\ell ,$ let $\alpha \left( \ell \right)
:=\tsum\nolimits_{j=1}^{\ell }\mathrm{v}_{\mathrm{p}}(1-\chi \left( \delta
\right) )^{j}$ and let $\alpha \left( 0\right) =0.$ Recall that $\alpha
\left( \ell \right) =0\ $for $\ell \leq p-2,$ while for an arbitrary $\ell ,$
$\alpha \left( \ell \right) =\tsum\nolimits_{n=0}^{\infty }\lfloor \frac{%
\ell }{p^{n}\left( p-1\right) }\rfloor \leq \lfloor \frac{\ell p}{\left(
p-1\right) ^{2}}\rfloor $ (cf. \cite[\S IV.1]{Ber03}). Let $\mathcal{S}$ be
a set of indeterminates and let $\pi $ be a distinguished indeterminate not
belonging to $\mathcal{S}.$ We denote$\ $by$\ M_{n}^{\mathcal{S}}\ $the
matrix ring $M_{n}(\mathcal{O}_{E}[[\pi ,\mathcal{S}]])^{\mid \tau \mid }.$
Recall that $k:=\max \{k_{i}\}\geq p.\ $For any integer $s\geq 0$\ we\ write 
$\vec{\pi}^{s}=\left( \pi ^{s},\pi ^{s},...,\pi ^{s}\right) ,$ and we denote
by $\overrightarrow{Id}$ the matrix of $M_{n}^{\mathcal{S}}$ whose
coordinates are the identity matrix. We need the following variant of (\cite[%
Lemma 10.3.2]{BB04}).

\begin{lemma}
\label{perturb lemma}For each $\gamma \in \Gamma _{K},$ let $G_{\gamma
}=G_{\gamma }\left( \mathcal{S}\right) \in \overrightarrow{Id}+\vec{\pi}%
M_{n}^{\mathcal{S}}.$ Let $c\geq 0$ be an integer and let $A=\left(
A_{1},A_{2},...,A_{f}\right) \ $be any matrix in $M_{n}\left( p^{c+\alpha
\left( k-1\right) }\mathcal{O}_{E}\right) ^{\mid \tau \mid }.$ There exists
a matrix%
\begin{equation*}
\hat{A}=\left( \hat{A}_{1},\hat{A}_{2},...,\hat{A}_{f}\right) \noindent \in
M_{n}\left( p^{c}\mathcal{O}_{E}[[\pi ,\mathcal{S}]]\right) ^{\mid \tau \mid
}
\end{equation*}%
such that:

\begin{enumerate}
\item[(i)] $\hat{A}\equiv A\func{mod}\ \vec{\pi};$

\item[(ii)] $\overrightarrow{Id}+\hat{A}\in \mathrm{GL}_{n}\left( \mathcal{O}%
_{E}[[\pi ,\mathcal{S}]]\right) ^{\mid \tau \mid };$

\item[(iii)] $\left( \overrightarrow{Id}+\hat{A}\right) \cdot G_{\gamma
}\cdot \gamma \left( \overrightarrow{Id}+\hat{A}\right) ^{-1}\equiv
G_{\gamma }\func{mod}\ \vec{\pi}^{k};$

\item[(iv)] If $A\in M_{n}\left( p^{1+\alpha \left( k-1\right) }\mathcal{O}%
_{E}\right) ^{\mid \tau \mid },$ then \noindent $\hat{A}\equiv 0\func{mod}\
p $ and 
\begin{equation*}
G_{\gamma }-\left( \overrightarrow{Id}+\hat{A}\right) \cdot G_{\gamma }\cdot
\gamma \left( \overrightarrow{Id}+\hat{A}\right) ^{-1}\noindent \equiv 0%
\func{mod}\ p.
\end{equation*}
\end{enumerate}
\end{lemma}

\begin{proof}
Since $k\geq p,\ $for any $A\in M_{n}\left( p^{c+\alpha \left( k-1\right) }%
\mathcal{O}_{E}\right) ^{\mid \tau \mid }$ it follows that $\overrightarrow{%
Id}+A\in \mathrm{GL}_{n}\left( \mathcal{O}_{E}\right) ^{\mid \tau \mid }.$
Let $\hat{A}_{i}=A_{i}+\pi \hat{A}_{i}^{1}+\pi ^{2}\hat{A}_{i}^{2}+\cdots
+\pi ^{k-1}\hat{A}_{i}^{k-1}$ and $G_{\gamma }=\left( G_{\gamma
}^{1},G_{\gamma }^{2},...,G_{\gamma }^{f}\right) ,$ with $G_{\gamma
}^{i}=Id+\pi G_{1}^{i}+\pi ^{2}G_{2}^{i}+\cdots +\pi
^{k-1}G_{k-1}^{i}+\cdots $ (suppressing the dependence of the $G_{j}^{i}\ $%
on $\gamma ).$ We first show that $\overrightarrow{Id}+\hat{A}\in \mathrm{GL}%
_{n}\left( \mathcal{O}_{E}[[\pi ,\mathcal{S}]]\right) ^{\mid \tau \mid }.$
Since $Id+A_{i}\in I_{n}+{M}_{n}\left( p^{c+\alpha \left( k-1\right) }%
\mathcal{O}_{E}\right) \ $for all $i,$ each coordinate matrix $Id+\hat{A}%
_{i} $ is invertible and the inverse is $\left( Id+\hat{A}_{i}\right)
^{-1}=\left( \tsum\nolimits_{n=0}^{\infty }\left( -1\right) ^{n}\pi
^{n}B_{i}^{n}\right) \left( Id+A_{i}\right) ^{-1},$ where $B_{i}=\left(
Id+A_{i}\right) ^{-1}\left( \hat{A}_{i}^{1}+\pi \hat{A}_{i}^{2}+\cdots +\pi
^{k-2}\hat{A}_{i}^{k-1}\right) .$ To prove part (iii), we need to choose the
matrices $\hat{A}_{i}^{j}\ $so that 
\begin{align*}
\left( Id+\pi G_{1}^{i}+\pi ^{2}G_{2}^{i}+\cdots +\pi
^{k-1}G_{k-1}^{i}+\cdots \right) & \gamma \left( A_{i}+\pi \hat{A}%
_{i}^{1}+\pi ^{2}\hat{A}_{i}^{2}+\cdots +\pi ^{k-1}\hat{A}_{i}^{k-1}\right) =
\\
\left( A_{i}+\pi \hat{A}_{i}^{1}+\pi ^{2}\hat{A}_{i}^{2}+\cdots +\pi ^{k-1}%
\hat{A}_{i}^{k-1}\right) & \left( Id+\pi G_{1}^{i}+\pi ^{2}G_{2}^{i}+\cdots
+\pi ^{k-1}G_{k-1}^{i}+\cdots \right) \func{mod}\ \pi ^{k}.
\end{align*}%
We may assume that $\gamma $ is a topological generator of $\Gamma _{K}.$ We
solve for the $\hat{A}_{i}^{j},$ bearing in mind that $\gamma \left( \pi
\right) ^{r}\equiv \chi \left( \gamma \right) ^{r}\pi ^{r}\func{mod}\ \pi
^{r+1}$ for all $r\geq 1.$ First, we solve for $\hat{A}_{i}^{1}\in
M_{n}\left( \mathcal{O}_{E}[[\mathcal{S}]]\right) $ so that $\left( 1-\chi
\left( \gamma \right) \right) \hat{A}_{i}^{1}=A_{i}G_{1}^{i}-G_{1}^{i}A_{i}.$
Since $A_{i}G_{1}^{i}-G_{1}^{i}A_{i}\in M_{n}\left( p^{c+\alpha \left(
k-1\right) }\mathcal{O}_{E}[[\mathcal{S}]]\right) ,$ we see that $\hat{A}%
_{i}^{1}:=\left( 1-\chi \left( \gamma \right) \right) ^{-1}\left(
A_{i}G_{1}^{i}-G_{1}^{i}A_{i}\right) $ and $\hat{A}_{i}^{1}\in M_{n}\left(
p^{c+\alpha \left( k-1\right) -\mathrm{v}_{\mathrm{p}}\left( 1-\chi \left(
\gamma \right) \right) }\mathcal{O}_{E}[[\mathcal{S}]]\right) .$ We then
solve for $\hat{A}_{i}^{2}$ so that $\left( \left( 1-\chi \left( \gamma
\right) ^{2}\right) \right) \hat{A}_{i}^{2}$ is an $\mathcal{O}_{E}$-linear
combination of products of $A_{i},\hat{A}_{i}^{1},G_{1}^{i},G_{2}^{i}$ which
belong to $M_{n}\left( p^{c+\alpha \left( k-1\right) -\mathrm{v}_{\mathrm{p}%
}\left( 1-\chi \left( \gamma \right) \right) }\mathcal{O}_{E}[[\mathcal{S}%
]]\right) .$ Dividing this linear combination by $\left( 1-\chi \left(
\gamma \right) ^{2}\right) ,$ we get%
\begin{equation*}
\hat{A}_{i}^{2}\in M_{n}\left( p^{c+\alpha \left( k-1\right) -\mathrm{v}_{%
\mathrm{p}}\left( \left( 1-\chi \left( \gamma \right) \right) \left( 1-\chi
\left( \gamma \right) ^{2}\right) \right) }\mathcal{O}_{E}[[\mathcal{S}%
]]\right) .
\end{equation*}%
Continuing this way we solve for $\hat{A}_{i}^{k-1}\in M_{n}\left( p^{c}%
\mathcal{O}_{E}[[\mathcal{S}]]\right) .$ Part (iv) is clear.
\end{proof}

\noindent Let $c_{i}\in \mathcal{O}_{E}^{\times }\ $and let $\Pi \left( 
\mathcal{S}\right) =\left( \Pi _{1},\Pi _{2},...,\Pi _{f}\right) \in M_{n}^{%
\mathcal{S}}$\ with $\det \left( \Pi _{i}\right) =c_{i}q^{k_{i}},\ $where $q=%
\frac{\left( 1+\pi \right) ^{p}-1}{\pi }.$\ We denote by$\ I$\ the ideal of $%
M_{n}\left( \mathcal{O}_{E}[[\mathcal{S}]]\right) $\ generated by the set $%
\{p\cdot Id,\ X_{i}\cdot Id:\ X_{i}\in \mathcal{S}\}\ $and by $\overline{%
M_{n}}$\ the quotient ring of $M_{n}\left( \mathcal{O}_{E}[[\mathcal{S}%
]]\right) \ $modulo $I.$ We also denote $P_{i}:=\Pi _{i}\func{mod}$ $I$ for
all $i=0,1,...,f-1.$\ Letting $\varphi $\ act trivially on the elements of $%
\mathcal{S},\mathcal{\ }$and letting $\varphi \left( \pi \right) =\left(
1+\pi \right) ^{p}-1,$\ we have%
\begin{equation}
\varphi \left( \vec{\alpha}\right) =(\varphi \left( \alpha _{1}\right)
,\varphi \left( \alpha _{2}\right) ,...,\varphi (\alpha _{0}))\noindent
\label{kellaki}
\end{equation}%
for all $\vec{\alpha}=\left( \alpha _{0},\alpha _{1},...,\alpha
_{f-1})\right) \in \mathcal{O}_{E}[[\pi ,\mathcal{S}]]^{\mid \tau \mid }.$\
We denote $\mathrm{Nm}_{\varphi }\left( \vec{\alpha}\right) :=\vec{\alpha}%
\cdot \varphi \left( \vec{\alpha}\right) \cdot \cdots \cdot \varphi
^{f-1}\left( \vec{\alpha}\right) $\ and for any matrix $A$\ in $M_{n}^{%
\mathcal{S}}$\ we write $\mathrm{Nm}_{\varphi }\left( A\right) :=A\cdot
\varphi \left( A\right) \cdot \cdots \cdot \varphi ^{f-1}\left( A\right) ,$\
with $\varphi $\ acting on each entry of the matrix $A$\ as in formula ($\ref%
{kellaki}).$ We fix a matrix $\Pi \left( \mathcal{S}\right) \in M_{n}^{%
\mathcal{S}}$ as above. For the rest of this section we assume that for any $%
\gamma \in \Gamma _{K}$ there exists a matrix $G_{\gamma }^{(k)}=G_{\gamma
}^{(k)}(\mathcal{S})\in M_{n}^{\mathcal{S}}$ such that:

(a)\noindent $\ G_{\gamma }^{(k)}(\mathcal{S})\equiv \overrightarrow{Id}%
\func{mod}\ \vec{\pi};$

(b)$\ G_{\gamma }^{(k)}(\mathcal{S})-\Pi (\mathcal{S})\mathcal{\varphi }%
(G_{\gamma }^{(k)}(\mathcal{S}))\gamma (\Pi (\mathcal{S})^{-1})\in \vec{\pi}%
^{k}M_{n}^{\mathcal{S}};$

(c)\ There exist no nonzero matrix $B\in M_{n}(\mathcal{O}_{E}[[\mathcal{S}%
]])^{\mid \tau \mid }$ and integer $t>0$ such that $BU=p^{ft}UB,$ \noindent\ 

\indent where $U=\mathrm{Nm}_{\varphi }\left( \Pi \left( \mathcal{S}\right)
\right) ;$

(d) If $k=k_{i}$ for all $i,$ we additionally assume that the operator%
\begin{equation}
\overline{H}\mapsto \overline{H-Q_{f}H(p^{fk}Q_{f}^{-1})}:\overline{M_{n}}%
\rightarrow \overline{M_{n}},  \label{tinasepwesena}
\end{equation}%
where $Q_{f}=P_{1}P_{2}\cdots P_{f-1}P_{0}\ $is surjective. Let $R_{\gamma
}^{\left( k\right) }\left( \mathcal{S}\right) $ be the matrices defined by%
\begin{equation}
\vec{\pi}^{k}R_{\gamma }^{\left( k\right) }\left( \mathcal{S}\right)
:=G_{\gamma }^{\left( k\right) }\left( \mathcal{S}\right) -\Pi \left( 
\mathcal{S}\right) \varphi \left( G_{\gamma }^{\left( k\right) }\left( 
\mathcal{S}\right) \right) \gamma \left( \Pi \left( \mathcal{S}\right)
^{-1}\right) .  \label{kati}
\end{equation}

\begin{proposition}
\label{perturb cor}Let $A=\left( A_{1},A_{2},...,A_{f}\right) \in
M_{n}\left( p^{\alpha \left( k-1\right) }\mathcal{O}_{E}\right) ^{\mid \tau
\mid }$ and let$\ Q_{f}^{A}:=\tprod\limits_{i=1}^{f}\left( Id+A_{i}\right)
P_{i},$ where $P_{i}=\Pi _{i}\func{mod}\ \pi $ for all $i.$ Assume that

\begin{enumerate}
\item[(1)] There exist no nonzero matrix $B\in M_{n}(\mathcal{O}_{E}[[%
\mathcal{S}]])^{\mid \tau \mid }$ and integer $t>0$ such that $%
BU_{A}=p^{ft}U_{A}B,$ where $U_{A}=\mathrm{Nm}_{\varphi }\left( \left( 
\overrightarrow{Id}+A\right) P\left( \mathcal{S}\right) \right) ,$ with $%
P\left( \mathcal{S}\right) =\Pi \left( \mathcal{S}\right) \func{mod}$ $I.$
\end{enumerate}

$~$If $n=2,$ we replace assumption $(1)$ by the following assumption.

\begin{enumerate}
\item[(\'{1})] $\mathrm{Tr}\left( Q_{f}^{A}\right) \not\in \overline{%
%TCIMACRO{\U{211a} }%
%BeginExpansion
\mathbb{Q}
%EndExpansion
_{p}}.$
\end{enumerate}

\noindent Let $\hat{A}\in M_{n}\left( \mathcal{O}_{E}[[\pi ,\mathcal{S}%
]]\right) ^{\mid \tau \mid }$ be as in Lemma $\ref{perturb lemma}\ $applied
for the matrices $G_{\gamma }\left( \mathcal{S}\right) :=G_{\gamma
}^{(k)}\left( \mathcal{S}\right) ,\ $where the $G_{\gamma }^{(k)}\left( 
\mathcal{S}\right) \ $are as in the assumptions preceding this proposition,$%
\ $and let $\Pi _{^{\hat{A}}}\left( \mathcal{S}\right) :=\left( 
\overrightarrow{Id}+\hat{A}\right) \Pi \left( \mathcal{S}\right) .$ Then for
each $\gamma \in \Gamma _{K}$ there exists a unique matrix $G_{\gamma ,\hat{A%
}}(\mathcal{S})\in M_{n}^{\mathcal{S}}\ $such that \noindent

\begin{enumerate}
\item[(i)] $G_{\gamma ,\hat{A}}(\mathcal{S})\equiv \overrightarrow{Id}\func{%
mod}\ \vec{\pi}$ and

\item[(ii)] $\Pi _{\hat{A}}(\mathcal{S})\mathcal{\varphi }(G_{\gamma ,\hat{A}%
}(\mathcal{S}))=G_{\gamma ,\hat{A}}(\mathcal{S})\gamma (\Pi _{\hat{A}}(%
\mathcal{S})).$
\end{enumerate}
\end{proposition}

\begin{proof}
Since $k\geq p,\ $it follows that $\det \left( Id+A_{i}\right) \in \mathcal{O%
}_{E}^{\times }\ $for all $i.$ Let $G_{\gamma ,\hat{A}}^{\left( k\right)
}\left( \mathcal{S}\right) :=G_{\gamma }^{\left( k\right) }\left( \mathcal{S}%
\right) $ for all $\gamma \in \Gamma _{K}.$ The proposition follows from 
\cite[Lemma 4.4]{DO10b}\ applied for the matrices $\Pi _{^{\hat{A}}}\left( 
\mathcal{S}\right) \ $and $G_{\gamma ,\hat{A}}^{\left( k\right) }\left( 
\mathcal{S}\right) \ $for the case where $\ell =k.$ Assumption $(1)$ of this
lemma clearly holds. If $k=k_{i}$ for all $i,$ since $k\geq p,\ $the operator%
\begin{equation}
\overline{H}\mapsto \overline{H-Q_{f}^{A}H(p^{fk}(Q_{f}^{A})^{-1})}:%
\overline{M_{n}}\rightarrow \overline{M_{n}}  \label{de mas xezeis}
\end{equation}%
coincides with the operator $(\ref{tinasepwesena})$ which was assumed to be
surjective. Hence assumption $(4)$ of \cite[Lemma 4.4]{DO10b}\ holds. If $%
n=2,$ assumption $(3)$ holds because of assumption $(\acute{1})$ and \cite[%
Corollary 5.3]{DO10b}. \noindent Lemma \ref{perturb lemma} (ii) implies\
that $Id+\hat{A}\in \mathrm{GL}_{n}\left( \mathcal{O}_{E}[[\pi ,\mathcal{S}%
]]\right) ^{\mid \tau \mid }.$ Moreover,%
\begin{eqnarray*}
&&G_{\gamma ,\hat{A}}^{\left( k\right) }\left( \mathcal{S}\right) -\Pi _{%
\hat{A}}\left( \mathcal{S}\right) \cdot \varphi \left( G_{\gamma ,\hat{A}%
}^{\left( k\right) }\left( \mathcal{S}\right) \right) \cdot \gamma \left(
\Pi _{\hat{A}}\left( \mathcal{S}\right) ^{-1}\right) = \\
&&G_{\gamma }^{\left( k\right) }\left( \mathcal{S}\right) -\left( Id+\hat{A}%
\right) \cdot \Pi \left( \mathcal{S}\right) \cdot \varphi \left( G_{\gamma
}^{\left( k\right) }\left( \mathcal{S}\right) \right) \cdot \gamma \Pi
\left( \mathcal{S}\right) ^{-1}\cdot \gamma \left( Id+\hat{A}\right) ^{-1}%
\overset{(\ref{kati})}{=} \\
&&G_{\gamma }^{\left( k\right) }\left( \mathcal{S}\right) -\left( Id+\hat{A}%
\right) \cdot G_{\gamma }^{\left( k\right) }\left( \mathcal{S}\right) \cdot
\gamma \left( Id+\hat{A}\right) ^{-1}+\vec{\pi}^{k}\cdot \left( Id+\hat{A}%
\right) \cdot R_{\gamma }^{\left( k\right) }\left( \mathcal{S}\right) \cdot
\gamma \left( Id+\hat{A}\right) ^{-1}.\ 
\end{eqnarray*}%
By Lemma \ref{perturb lemma} (iii), $G_{\gamma }^{\left( k\right) }\left( 
\mathcal{S}\right) -\left( Id+\hat{A}\right) \cdot G_{\gamma }^{\left(
k\right) }\left( \mathcal{S}\right) \cdot \gamma \left( Id+\hat{A}\right)
^{-1}\in \vec{\pi}^{k}M_{n}^{\mathcal{S}}$ and therefore assumption $(2)$ of 
\cite[Lemma 4.4]{DO10b} holds. This completes the proof.\noindent
\end{proof}

\begin{proposition}
\label{gamma acts}For any $\vec{a}=(a_{0},a_{1},...,a_{f-1})\in \mathfrak{m}%
_{E}^{\left\vert \mathcal{S}\right\vert }$ and any $\gamma _{1},\gamma
_{2},\gamma \in \Gamma _{K},$ the following equations hold:

\begin{enumerate}
\item[(i)\noindent ] $G_{\gamma _{1}\gamma _{2},\hat{A}}(\vec{a})=G_{\gamma
_{1},\hat{A}}(\vec{a})\gamma _{1}(G_{\gamma _{2},\hat{A}}(\vec{a}))$ and

\item[(ii)\noindent ] $\Pi _{\hat{A}}(\vec{a})\mathcal{\varphi }(G_{\gamma ,%
\hat{A}}(\vec{a}))=G_{\gamma ,\hat{A}}(\vec{a})\gamma (\Pi _{\hat{A}}(\vec{a}%
)).$
\end{enumerate}
\end{proposition}

\begin{proof}
Both matrices $G_{\gamma _{1}\gamma _{2},\hat{A}}(\mathcal{S})$ and $%
G_{\gamma _{1},\,\hat{A}}(\mathcal{S})\gamma _{1}(G_{\gamma _{2},\hat{A}}(%
\mathcal{S}))$ are $\equiv \overrightarrow{Id}\func{mod}\ \vec{\pi}$ and are
solutions in $B$ of the equation $\Pi (\mathcal{S})\mathcal{\varphi }%
(B)=B\gamma (\Pi (\mathcal{S})).$ They are equal by the uniqueness part of
Proposition \ref{perturb cor}. The second equation follows from conclusion
(ii) of the same proposition. $~$
\end{proof}

For any $\vec{a}\in \mathfrak{m}_{E}^{\left\vert \mathcal{S}\right\vert }$
we equip$\ \mathbb{N}_{\hat{A}}\left( \vec{a}\right)
=\tbigoplus\nolimits_{i=1}^{n}\left( \mathcal{O}_{E}[[\pi ]]^{\mid \tau \mid
}\right) \eta _{i}$ with $\varphi $ and $\Gamma _{K}$-actions defined by $%
\left( \varphi \left( \eta _{1}\right) ,\varphi \left( \eta _{2}\right)
,...,\varphi \left( \eta _{n}\right) \right) =\left( \eta _{1},\eta
_{2},...,\eta _{n}\right) \Pi _{\hat{A}}(\vec{a})$ and $\left( \gamma \eta
_{1},\gamma \eta _{2},...,\gamma \eta _{n}\right) =\left( \eta _{1},\eta
_{2},...,\eta _{n}\right) G_{\gamma ,\hat{A}}(\vec{a})$ respectively.
Proposition \ref{gamma acts} implies that $(\gamma _{1}\gamma _{2})x$%
\noindent $=\gamma _{1}(\gamma _{2}x)$ and $\varphi (\gamma x)=\gamma
(\varphi (x))$ for all $x\in \mathbb{N}_{\hat{A}}(\vec{a})\ $and $\gamma
,\gamma _{1},\gamma _{2}\in \Gamma _{K}.$ Since $G_{\gamma ,\hat{A}}(\vec{a}%
)\equiv \overrightarrow{Id}\func{mod}\ \vec{\pi},$ it follows that the $%
\Gamma _{K}$ action on $\mathbb{N}_{\hat{A}}(\vec{a})\ $is trivial modulo $%
\pi \mathbb{N}_{\hat{A}}(\vec{a}).$ We have the following.

\begin{proposition}
\label{rank two wach modules construction copy(1)}For any $\vec{a}\in 
\mathfrak{m}_{E}^{\left\vert \mathcal{S}\right\vert }\ $the module $\mathbb{N%
}_{\hat{A}}(\vec{a})$ equipped with the $\varphi $ and $\Gamma _{K}$-actions
defined by $\Pi _{\hat{A}}(\vec{a})$ and $G_{\gamma ,\hat{A}}(\vec{a})$
respectively is a Wach module corresponding to some $G_{K}$-stable $\mathcal{%
O}_{E}$-lattice inside some $n$-dimensional crystalline $E$-representation
of $G_{K}$ with Hodge-Tate weights in $[-k;\ 0].$

\begin{proof}
By Theorem \ref{berger thm}, the only thing left to prove is that $q^{k}%
\mathbb{N}_{\hat{A}}(\vec{a})\boldsymbol{\subset }\varphi ^{\ast }(\mathbb{N}%
_{\hat{A}}(\vec{a})).$ This is identical to the proof of \cite[Proposition
4.6]{DO10b}.\noindent
\end{proof}
\end{proposition}

Let $V_{\hat{A}}(\vec{a})=E\otimes _{\mathcal{O}_{E}}\mathrm{T}_{\hat{A}}(%
\vec{\alpha}),$ where $\mathrm{T}_{\hat{A}}(\vec{\alpha})=\mathbb{T}(\mathbb{%
D}_{\hat{A}}(\vec{a}))$ and $\mathbb{D}_{\hat{A}}(\vec{a})=\mathbb{A}%
_{K,E}\otimes _{\mathbb{A}_{K,E}^{+}}\mathbb{N}_{\hat{A}}(\vec{a}).$ By
Theorem \ref{berger thm} the representations $V_{\hat{A}}(\vec{a})$ are $n$%
-dimensional crystalline $E$-representations of $G_{K}$ with Hodge-Tate
weights in $[-k;\ 0].$ They are independent of the lifting $\hat{A}$ of $A$
(see Remark \ref{oneremark...} (i) below) and instead we denote them by $%
V_{A}(\vec{a}).$ Regarding the $\func{mod}$ $p$ reductions of these
representations, we have the following.

\begin{proposition}
\label{xwris A=0}For any $\vec{a}\in \mathfrak{m}_{E}^{\left\vert \mathcal{S}%
\right\vert }\ $there exist $G_{K}$-stable $\mathcal{O}_{E}$-lattices such
that$\ \overline{V}_{A}(\vec{a})\simeq \overline{V}_{A}(\vec{0}).$

\begin{proof}
Identical to the proof of \cite[Theorem 4.7 ]{DO10b}, given that $\Pi _{\hat{%
A}}(\vec{a})\equiv \Pi _{\hat{A}}(\vec{0})\func{mod}$ $\mathfrak{m}_{E}$ and 
$G_{\gamma ,\hat{A}}(\vec{a})\equiv G_{\gamma ,\hat{A}}(\vec{0})\func{mod}$ $%
\mathfrak{m}_{E}.$
\end{proof}
\end{proposition}

\begin{lemma}
\label{perturb subcol}Let $A\in M_{n}\left( p^{1+\alpha \left( k-1\right) }%
\mathcal{O}_{E}\right) ^{\mid \tau \mid }.\ $For any $\gamma \in \Gamma _{K}$
and for $\dag \in \{0,\hat{A}\},$ let%
\begin{equation}
G_{\gamma ,\dag }^{\left( k\right) }\left( \mathcal{S}\right) -\Pi _{\dag
}\left( \mathcal{S}\right) \varphi \left( G_{\gamma ,\dag }^{\left( k\right)
}\left( \mathcal{S}\right) \right) \gamma \left( \Pi _{\dag }\left( \mathcal{%
S}\right) ^{-1}\right) =:\vec{\pi}^{k}R_{\gamma ,\dag }^{\left( k\right) },\ 
\mathrm{where}\ G_{\gamma ,\hat{A}}^{\left( k\right) }\left( \mathcal{S}%
\right) :=G_{\gamma }^{\left( k\right) }\left( \mathcal{S}\right) .
\label{stellaki}
\end{equation}%
Then (i) $\Pi _{\hat{A}}(\mathcal{S})\equiv \Pi (\mathcal{S})\func{mod}\ I,\ 
$and (ii) $R_{\gamma ,\hat{A}}^{\left( k\right) }(\mathcal{S})\equiv
R_{\gamma }^{\left( k\right) }(\mathcal{S})\func{mod}\ I.$

\begin{proof}
By Lemma \ref{perturb lemma} (iv), $\hat{A}\equiv 0\func{mod}$ $p\ $and part
(i)$\ $is clear. The lemma follows immediately reducing equations (\ref%
{stellaki}) $\func{mod}$ $I.$
\end{proof}
\end{lemma}

\section{Proof of the Theorem\label{tinasepw...}}

\noindent Throughout this section we assume that $n=2.$ For the rest of the
paper we fix an $f$-tuple $P=\left( P_{1},P_{2},...,P_{f}\right) $ chosen as
in $\S \ref{dof}.$\ For any $A\in M_{2}\left( \mathcal{O}_{E}[[\pi ]]\right)
^{\mid \tau \mid }$ we define $Q_{f}^{A}:=\tprod\nolimits_{i=1}^{i=f}\left(
Id+A_{i}\right) P_{i}$\ \ and we let $Q_{f}:=Q_{f}^{0}.$ We define%
\begin{equation*}
\ m_{k}:=\left\{ 
\begin{array}{l}
\ \ 0\ \ \ \ \ \mathrm{if}\ \text{\ }k_{j}=p\ \mathrm{for\ all\ }j\ \mathrm{%
and\ Tr}\left( Q_{f}\right) \not\in \overline{%
%TCIMACRO{\U{211a} }%
%BeginExpansion
\mathbb{Q}
%EndExpansion
_{p}}, \\ 
\lfloor \frac{k-1}{p-1}\rfloor \ \mathrm{otherwise.}%
\end{array}%
\right.
\end{equation*}%
For matrices $P$ chosen as in $\S \ref{dof}$ the condition $\mathrm{Tr}%
\left( Q_{f}\right) \not\in \overline{%
%TCIMACRO{\U{211a} }%
%BeginExpansion
\mathbb{Q}
%EndExpansion
_{p}}$ in the definition of $m_{k}$ turns out to be redundant (see Lemma \ref%
{RMK} (i) below), and $m_{k}$ coincides with the integer $m$ defined in
formula (\ref{m}). Let $\Pi \left( \mathcal{S}\right) =(\Pi _{1}\left(
S_{1}\right) ,\Pi _{2}\left( S_{2}\right) ,\noindent ...,\noindent \Pi
_{f}\left( S_{f}\right) )\in M_{n}^{\mathcal{S}},$ where $\Pi _{i}$\ are
matrices of one of the following four types: 
\begin{equation*}
t_{1}\mathbf{:}\mathbf{\ }\left( 
\begin{array}{cc}
c_{i}q^{k_{i}} & 0 \\ 
S_{i}\varphi (z_{i}) & 1%
\end{array}%
\right) ,\ t_{2}\mathbf{:\ }\left( 
\begin{array}{cc}
S_{i}\varphi (z_{i}) & 1 \\ 
c_{i}q^{k_{i}} & 0%
\end{array}%
\right) ,\ t_{3}\mathbf{:}\mathbf{\ }\left( 
\begin{array}{cc}
1 & S_{i}\varphi (z_{i}) \\ 
0 & c_{i}q^{k_{i}}%
\end{array}%
\right) ,\ t_{4}\mathbf{:\ }\left( 
\begin{array}{cc}
0 & c_{i}q^{k_{i}} \\ 
1 & S_{i}\varphi (z_{i})%
\end{array}%
\right) ,
\end{equation*}%
with$\ S_{i}\in \mathcal{S\ }$and $c_{i}\in \mathcal{O}_{E}^{\times }.\ $The 
$z_{i}$\ are polynomials in $%
%TCIMACRO{\U{2124} }%
%BeginExpansion
\mathbb{Z}
%EndExpansion
_{p}[\pi ]\ $of degree $\leq k-1\ $such that $z_{i}\equiv p^{m_{k}}\func{mod}%
~\pi ,$\ suitably chosen so that there exist matrices \ $G_{\gamma }^{(k)}(%
\mathcal{S})\in M_{n}^{\mathcal{S}}$ with $G_{\gamma }^{(k)}(\mathcal{S}%
)\equiv \overrightarrow{Id}$ $\func{mod}$ $\vec{\pi}$ such that $G_{\gamma
}^{(k)}(\mathcal{S})-\Pi (\mathcal{S})\varphi (G_{\gamma }^{(k)}(\mathcal{S}%
))\gamma (\Pi (\mathcal{S})^{-1})\in \vec{\pi}^{k}M_{n}^{\mathcal{S}}.$ The
existence of such polynomials has been established in \cite[Proposition 5.9
\& Remark 5.12]{DO10b}. We let $\vec{X}=\left( X_{1},X_{2},...,X_{f}\right) $
with $\vec{X}=p^{m_{k}}\vec{S},$ and we choose $\Pi $ so that its modulo $%
\vec{\pi}$ reduction equals $P.$ In particular, the type of $\Pi _{i}$
coincides with the type of $P_{i}$ for all $i.$ If $A\in M_{2}\left( 
\mathcal{O}_{E}[[\pi ]]\right) ^{\mid \tau \mid },$ let $\hat{A}$ be a fixed
choice of a lifting of $A$ as in Lemma \ref{perturb lemma} with respect to a
fixed choice of matrices $G_{\gamma }(\mathcal{S}):=G_{\gamma }^{(k)}(%
\mathcal{S})$ as above, and let $\Pi _{\hat{A}}\left( \mathcal{S}\right)
=\left( Id+\hat{A}\right) \Pi (\mathcal{S}).$ Let $E_{ij},$ $i,j=1,2,$ be
the $2\times 2$ matrix with $\left( i,j\right) \ $entry $1$ and all other
entries $0.$

\begin{lemma}

\begin{enumerate}
\item[(i)] \label{RMK}\ \textrm{$Tr$}$\left( Q_{f}\right) \not\in \overline{%
%TCIMACRO{\U{211a} }%
%BeginExpansion
\mathbb{Q}
%EndExpansion
_{p}}\ ;$

\item[(ii)] \textrm{$Tr$}$\left( Q_{f}\right) \not\in p\overline{%
%TCIMACRO{\U{2124} }%
%BeginExpansion
\mathbb{Z}
%EndExpansion
_{p}}[X_{1},X_{2},...,X_{f}].$

\item[(iii)] For any $A\in M_{2}\left( p\mathcal{O}_{E}[[\pi ]]\right)
^{\mid \tau \mid },\ $\textrm{$Tr$}$(Q_{f}^{A})\not\in \overline{%
%TCIMACRO{\U{211a} }%
%BeginExpansion
\mathbb{Q}
%EndExpansion
_{p}}\ ;$

\item[(iv)] For any $A\in M_{2}\left( p\mathcal{O}_{E}[[\pi ]]\right) ^{\mid
\tau \mid }$ the operator ($\ref{de mas xezeis})$ is surjective.
\end{enumerate}

\begin{proof}
For part (i) recall that in the proofs of \cite[Theorems $1.5$ \& $1.7$]%
{DO10b}, the types of the coordinate matrices $P_{i}$ of $P$ have been
chosen so that \textrm{$Tr$}$\left( Q_{f}\right) \not\in \overline{%
%TCIMACRO{\U{211a} }%
%BeginExpansion
\mathbb{Q}
%EndExpansion
_{p}}.$ For part (ii), we have 
\begin{equation}
P_{i}\func{mod}\ p=\left\{ 
\begin{array}{c}
c\left( k_{i}\right) E_{11}+E_{22}+X_{i}E_{21}\ \text{if }P_{i}=t_{1}, \\ 
c\left( k_{i}\right) E_{21}+E_{12}+X_{i}E_{11}\ \text{if }P_{i}=t_{2}, \\ 
c\left( k_{i}\right) E_{22}+E_{11}+X_{i}E_{12}\ \text{if }P_{i}=t_{3}, \\ 
c\left( k_{i}\right) E_{12}+E_{21}+X_{i}E_{22}\ \text{if }P_{i}=t_{4},%
\end{array}%
\right. \ \text{where}\ c\left( k_{i}\right) =\left\{ 
\begin{array}{c}
0\ \text{if }k_{i}>0, \\ 
\\ 
1\ \text{if }k_{i}=0.%
\end{array}%
\right.  \label{C(K_i)}
\end{equation}%
The $\left( i,i\right) $ entries in $Q_{f}\func{mod}$ $p$ are sums of
distinct terms of the form $1$ and $X_{i_{1}}\cdot X_{i_{2}}\cdot \cdots
\cdot X_{i_{r_{i}}}$ for some $1\leq r_{i}\leq f.$ Hence $\mathrm{Tr}\left(
Q_{f}\right) \not\equiv 0\func{mod}$ $p\ $(if the diagonal entries of $%
\mathrm{Tr}\left( Q_{f}\right) \func{mod}$ $p$ coincide, we use that $p\neq
2 $). For part (iii), assume that \textrm{$Tr$}$(Q_{f}^{A})\in \overline{%
%TCIMACRO{\U{211a} }%
%BeginExpansion
\mathbb{Q}
%EndExpansion
_{p}}.$ Since the entries of $Q_{f}^{A}$ are in $\overline{%
%TCIMACRO{\U{2124} }%
%BeginExpansion
\mathbb{Z}
%EndExpansion
_{p}}[X_{1},X_{2},...,X_{f}]$ it follows that \textrm{$Tr$}$(Q_{f}^{A})\in 
\overline{%
%TCIMACRO{\U{2124} }%
%BeginExpansion
\mathbb{Z}
%EndExpansion
_{p}}.$ Since $Q_{f}^{A}\equiv Q_{f}\func{mod}$ $p$ it follows that \textrm{$%
Tr$}$(Q_{f}^{A})\equiv \mathrm{Tr}\left( Q_{f}\right) \func{mod}$ $p$ and
therefore that $\mathrm{Tr}\left( Q_{f}\right) \in \overline{%
%TCIMACRO{\U{2124} }%
%BeginExpansion
\mathbb{Z}
%EndExpansion
_{p}}+p\overline{%
%TCIMACRO{\U{2124} }%
%BeginExpansion
\mathbb{Z}
%EndExpansion
_{p}}[X_{1},X_{2},...,X_{f}].$ Since $\mathrm{Tr}\left( Q_{f}\right) \not\in 
\overline{%
%TCIMACRO{\U{211a} }%
%BeginExpansion
\mathbb{Q}
%EndExpansion
_{p}},$ \cite[Lemma 5.19 \& Corollary 5.17]{DO10b} imply that $Q_{f}\func{mod%
}\ I=E_{12}$ or $E_{21},$ therefore $\mathrm{Tr}\left( Q_{f}\right) \equiv 0%
\func{mod}$ $\left( p,X_{1},...,X_{f}\right) .$ Hence $\mathrm{Tr}\left(
Q_{f}\right) \in \left( \overline{%
%TCIMACRO{\U{2124} }%
%BeginExpansion
\mathbb{Z}
%EndExpansion
_{p}}+p\overline{%
%TCIMACRO{\U{2124} }%
%BeginExpansion
\mathbb{Z}
%EndExpansion
_{p}}[X_{1},X_{2},...,X_{f}]\right) \tbigcap \left( p,X_{1},...,X_{f}\right)
=p\overline{%
%TCIMACRO{\U{2124} }%
%BeginExpansion
\mathbb{Z}
%EndExpansion
_{p}}[X_{1},X_{2},...,X_{f}]$ which contradicts part (ii) of the lemma. Part
(iv) for $A=0$ follows from \cite[Corollary 5.20]{DO10b}. The general case
holds because the operators with any $A\in M_{2}\left( p\mathcal{O}_{E}[[\pi
]]\right) ^{\mid \tau \mid }$ coincide with that with $A=0.$
\end{proof}
\end{lemma}

\begin{proposition}
Let $A\in M_{2}\left( p^{\alpha \left( k-1\right) }\mathcal{O}_{E}[[\pi
]]\right) ^{\mid \tau \mid }\ $and let $\Pi _{\hat{A}}\left( \mathcal{S}%
\right) $ be as in the beginning of $\S \ref{tinasepw...}.$ For each $\gamma
\in \Gamma _{K}$ there exists a unique matrix $G_{\gamma }\left( \mathcal{S}%
\right) \in M_{2}^{\mathcal{S}}$ such that

\begin{enumerate}
\item[(i)] $G_{\gamma }\left( \mathcal{S}\right) \equiv \overrightarrow{Id}%
\func{mod}$ $\vec{\pi};$

\item[(ii)] $\Pi _{\hat{A}}\left( \mathcal{S}\right) \varphi \left(
G_{\gamma }\left( \mathcal{S}\right) \right) =G_{\gamma }\left( \mathcal{S}%
\right) \gamma \Pi _{\hat{A}}\left( \mathcal{S}\right) $ for all $\gamma .$
\end{enumerate}

\begin{proof}
Conditions (a) and (b) preceding Proposition \ref{perturb cor} hold by the
discussion in the beginning of $\S \ref{tinasepw...}.$ Condition (c)
preceding Proposition \ref{perturb cor} and Condition $(\acute{1})$ of
Proposition \ref{perturb cor} hold because \textrm{$Tr$}$\left( Q_{f}\right)
\not\in \overline{%
%TCIMACRO{\U{211a} }%
%BeginExpansion
\mathbb{Q}
%EndExpansion
_{p}}$ and \textrm{$Tr$}$(Q_{f}^{A})\not\in \overline{%
%TCIMACRO{\U{211a} }%
%BeginExpansion
\mathbb{Q}
%EndExpansion
_{p}}$ respectively, by Lemma \ref{RMK} (i) \& (iii). Finally, Condition (d)
preceding Proposition \ref{perturb cor} holds by Lemma \ref{RMK} (iv) with $%
A=0.$ The proposition follows by Proposition \ref{perturb cor}.
\end{proof}
\end{proposition}

For any $\vec{a}\in \mathfrak{m}_{E}^{\left\vert \mathcal{S}\right\vert }$
and $\dag \in \{0,\hat{A}\},\ $we equip$\ \mathbb{N}_{\dag }\left( \vec{a}%
\right) =\left( \mathcal{O}_{E}[[\pi ]]^{\mid \tau \mid }\right) \eta
_{1}\oplus \left( \mathcal{O}_{E}[[\pi ]]^{\mid \tau \mid }\right) \eta _{2}$
with the $\varphi $ and $\Gamma _{K}$-actions defined defined by $\left(
\varphi \left( \eta _{1}\right) ,\varphi \left( \eta _{2}\right) \right)
=\left( \eta _{1},\eta _{2}\right) \Pi _{\dag }(\vec{a})$ and $\left( \gamma
\eta _{1},\gamma \eta _{2}\right) =\left( \eta _{1},\eta _{2}\right)
G_{\gamma ,\dag }(\vec{a})$ respectively.

\begin{corollary}
\label{rank two wach modules construction}The module $\mathbb{N}_{\dag }(%
\vec{a})$ with the above $\varphi $ and $\Gamma _{K}$ actions is a Wach
module corresponding to some $G_{K}$-stable $\mathcal{O}_{E}$-lattice of a $%
2 $-dimensional crystalline $E$-representation $V_{\dag }(\vec{a})$ of $%
G_{K} $ with Hodge-Tate weights in $[-k;\ 0].$

\begin{proof}
Follows immediately from Proposition \ref{rank two wach modules construction
copy(1)}.
\end{proof}
\end{corollary}

As in \S $\ref{newfromold},$ the representation $V_{\hat{A}}(\vec{a})$ is
independent of the lifting $\hat{A}$ and we simply write $V_{A}(\vec{a}).$

\begin{lemma}
\label{daggermodp lemma}If $s\geq k+1\ $and $B\in M_{2}\left( \mathcal{O}%
_{E}[[\mathcal{S}]]\right) ^{\mid \tau \mid }$ is such that $B\equiv
Q_{f}B\left( p^{f\left( s-1\right) }Q_{f}^{-1}\right) \func{mod}\ I,$ then $%
B\equiv 0\func{mod}\ I.$
\end{lemma}

\begin{proof}
We may assume that $s-1=k=k_{i}\ $for all $i,$ otherwise $Q_{f}B\left(
p^{f\left( s-1\right) }Q_{f}^{-1}\right) \equiv 0\func{mod}\ I$ and the
lemma holds trivially. By Lemma \ref{RMK} (i), $\mathrm{Tr}\left(
Q_{f}\right) \not\in \overline{%
%TCIMACRO{\U{211a} }%
%BeginExpansion
\mathbb{Q}
%EndExpansion
_{p}}$ and \cite[Lemma 5.19 \& Corollary 5.17]{DO10b} (where in \cite{DO10b} 
$\overline{Q_{f}}:=Q_{f}\func{mod}\ I$) imply that $Q_{f}\func{mod}\
I=E_{12} $ or $E_{21}.$ Recall that $k\geq p.$

\begin{claimA}
If $Q_{f}\func{mod}\ I=E_{ij}$ with $i\neq j$ then $p^{fk}Q_{f}^{-1}\func{mod%
}\ I=-Q_{f}\func{mod}\ I.$ If $Q_{f}\func{mod}\ I=E_{11},$ then $%
p^{fk}Q_{f}^{-1}\func{mod}\ I=E_{22},$ and if $Q_{f}\func{mod}\ I=E_{22}$
then $p^{fk}Q_{f}^{-1}\func{mod}\ I=E_{11}.$
\end{claimA}

\begin{proof}[Proof of Claim.]
By induction on $f.$ For $f=1,$ formula $(\ref{C(K_i)})$ becomes 
\begin{equation*}
P\func{mod}\ I=\left\{ 
\begin{array}{c}
E_{22}\ \mathrm{if}\text{ }P=t_{1}, \\ 
E_{12}\ \mathrm{if}\text{\ }P=t_{2}, \\ 
E_{11}\ \mathrm{if}\text{ }P=t_{3}, \\ 
E_{21}\ \mathrm{if}\text{ }P=t_{4},%
\end{array}%
\right.
\end{equation*}%
and the claim is clear. Suppose $f\geq 2.$ Case (i). $Q_{f}\func{mod}\
I=E_{12}.$ If $P_{1}P_{2}\cdots P_{f-1}\func{mod}$ $I=E_{11}$ then $P_{f}%
\func{mod}$ $I=E_{12}.$ The matrix $P_{f}$ is of type $2$ and by the
inductive hypothesis%
\begin{equation*}
p^{kf}Q_{f}^{-1}\func{mod}\ I=\left( p^{k}P_{f}^{-1}\right) \cdot \left(
\left( p^{k}P_{f-1}^{-1}\right) \cdot \cdots \cdot \left(
p^{k}P_{1}^{-1}\right) \right) \func{mod}\ I=-E_{12}\cdot E_{22}=-E_{12}.
\end{equation*}%
If $P_{1}P_{2}\cdots P_{f-1}\func{mod}$ $I=E_{12}$ then $P_{f}\func{mod}$ $%
I=E_{22}.$ The matrix $P_{f}$ is of type $1$ and by the inductive hypothesis%
\begin{equation*}
p^{kf}Q_{f}^{-1}\func{mod}\ I=\left( p^{k}P_{f}^{-1}\right) \cdot \left(
\left( p^{k}P_{f-1}^{-1}\right) \cdot \cdots \cdot \left(
p^{k}P_{1}^{-1}\right) \right) \func{mod}\ I=E_{11}\cdot \left(
-E_{12}\right) =-E_{12}.
\end{equation*}%
The claim follows by the inductive hypothesis, arguing similarly for the
other possibilities for $Q_{f}\func{mod}\ I.$
\end{proof}

\noindent The Claim combined with the formula $B\equiv Q_{f}B\left(
p^{fk}Q_{f}^{-1}\right) \func{mod}\ I$ imply that $B\equiv -E_{ij}BE_{ij}%
\func{mod}$ $I$ with $i\neq j.$ From the latter it is immediate that $%
B\equiv 0\func{mod}$ $I.$
\end{proof}

\begin{proposition}
\label{dagger}Let $A\in M_{2}\left( p^{1+\alpha\left( k-1\right) }\mathcal{O}%
_{E}[[\pi]]\right) ^{\mid\tau\mid}$ and let $\Pi_{\hat{A}}(\mathcal{S}) $
and $G_{\gamma,\hat{A}}(\mathcal{S})$ be as in Proposition $\ref{perturb cor}%
.$ Then $\Pi_{\hat{A}}(\mathcal{S})\equiv\Pi(\mathcal{S})\func{mod}\ I $ and 
$G_{\gamma,\hat{A}}(\mathcal{S})\equiv G_{\gamma }(\mathcal{S})\func{mod}\
I. $
\end{proposition}

\begin{proof}
By Lemma \ref{perturb lemma} (iv), $\hat{A}\equiv0\func{mod}\ I,$ hence $%
\Pi_{\hat{A}}(\mathcal{S})\equiv\Pi(\mathcal{S})\func{mod}\ I.\ $Fix a
topological generator $\gamma$ of $\Gamma_{K}.$ By the proofs of \cite[%
Propositions 5.9 \& 5.11]{DO10b}, there exists a matrix $G_{\gamma }^{(k)}(%
\mathcal{S})\in M_{n}^{\mathcal{S}}$ with $G_{\gamma}^{(k)}(\mathcal{S}%
)\equiv\overrightarrow{Id}\func{mod}$ $\vec{\pi}$ and a matrix $R^{\left(
k\right) }\left( \mathcal{S}\right) \in$ $M_{n}^{\mathcal{S}}$ such that%
\begin{equation}
G_{\gamma}^{(k)}(\mathcal{S})-\Pi(\mathcal{S})\cdot\varphi\left( G_{\gamma
}^{(k)}(\mathcal{S})\right) \cdot\gamma\left( \Pi(\mathcal{S})^{-1}\right) =%
\vec{\pi}^{k}R^{\left( k\right) }\left( \mathcal{S}\right) .
\label{pesmekati...}
\end{equation}
Moreover, by the proof of \cite[Lemma 4.1]{DO10b}, for all $s\geq k+1$ there
exist matrices $G_{\gamma}^{(s)}(\mathcal{S})\in M_{n}^{\mathcal{S}}$ and $%
R^{\left( s\right) }\left( \mathcal{S}\right) \in$ $M_{n}^{\mathcal{S}}$
such that $G_{\gamma}^{(s)}(\mathcal{S})\equiv G_{\gamma}^{(s-1)}(\mathcal{S}%
)\func{mod}$ $\vec{\pi}^{s-1}M_{n}^{\mathcal{S}}$ and%
\begin{equation}
G_{\gamma}^{(s)}(\mathcal{S})-\Pi(\mathcal{S})\cdot\varphi\left( G_{\gamma
}^{(s)}(\mathcal{S})\right) \cdot\gamma\left( \Pi(\mathcal{S})^{-1}\right) =%
\vec{\pi}^{s}R^{\left( s\right) }\left( \mathcal{S}\right) .  \label{sofoula}
\end{equation}
Arguing as in the proof of part (ii) of Proposition \ref{perturb cor} and
taking into account equations (\ref{pesmekati...}) \& (\ref{sofoula}) we see
that for all $s\geq k$ there exist matrices $R_{\hat{A}}^{\left( s\right)
}\left( \mathcal{S}\right) \in M_{n}^{\mathcal{S}}$ such that such that 
\begin{equation}
G_{\gamma,\hat{A}}^{(s)}(\mathcal{S})-\Pi_{\hat{A}}(\mathcal{S})\cdot
\varphi\left( G_{\gamma,\hat{A}}^{(s)}(\mathcal{S})\right) \cdot
\gamma\left( \Pi_{\hat{A}}(\mathcal{S})^{-1}\right) =\vec{\pi}^{s}R_{\hat {A}%
}^{\left( s\right) }\left( \mathcal{S}\right)  \label{einaiswsto}
\end{equation}

\noindent Combining equations (\ref{pesmekati...}), (\ref{sofoula}) and (\ref%
{einaiswsto}) for all $s\geq k$ we write%
\begin{equation}
G_{\gamma,\dag}^{(s)}(\mathcal{S})-\Pi_{\dag}(\mathcal{S})\cdot\varphi\left(
G_{\gamma,\dag}^{(s)}(\mathcal{S})\right) \cdot\gamma\left( \Pi_{\dag }(%
\mathcal{S})^{-1}\right) =\vec{\pi}^{s}R_{\dag}^{\left( s\right) }\left( 
\mathcal{S}\right) ,  \label{einai swsto!}
\end{equation}
with$\mathrm{\ }\dag\in\{0,\hat{A}\}.$ We defined $G_{\gamma,\hat{A}%
}^{\left( k\right) }\left( \mathcal{S}\right) :=G_{\gamma}^{\left( k\right)
}\left( \mathcal{S}\right) $ and by Lemma\ \ref{perturb subcol} (ii), $%
R_{\gamma,\hat{A}}^{\left( k\right) }(\mathcal{S})\equiv R_{\gamma }^{\left(
k\right) }(\mathcal{S})\func{mod}\ I.$ We will show by induction that $%
G_{\gamma,\hat{A}}^{(s)}(\mathcal{S})\equiv G_{\gamma}^{(s)}(\mathcal{S})%
\func{mod}\ I$ and $R_{\hat{A}}^{\left( s\right) }\left( \mathcal{S}\right)
\equiv R^{\left( s\right) }\left( \mathcal{S}\right) \func{mod}\ I$ for all $%
s\geq k.$ For $s\geq k+1,$ let $G_{\gamma,\dag}^{(s)}=G_{\gamma,%
\dag}^{(s-1)}+\vec{\pi}^{s-1}H_{\dag }^{(s)},$ where $H_{\dag}^{\left(
s\right) }=H_{\gamma,\dag}^{\left( s\right) }\in M_{n}(\mathcal{O}_{E}[[%
\mathcal{S}]])^{\mid\tau\mid},$ and let$\ R_{\dag}^{(s)}\left( \mathcal{S}%
\right) =\overline{R}_{\dag}^{(s)}\left( \mathcal{S}\right) +\vec{\pi}\cdot
C_{\dag}^{\left( s\right) }\ $for some matrices $\overline{R}%
_{\dag}^{(s)}\left( \mathcal{S}\right) \in M_{n}(\mathcal{O}_{E}[[\mathcal{S}%
]])^{\mid\tau\mid}$ and $C_{\dag }^{\left( s\right) }\in M_{n}^{\mathcal{S}}$%
$.$ By the inductive hypothesis, $\overline{R}_{\hat{A}}^{(s-1)}\left( 
\mathcal{S}\right) +\vec{\pi}\cdot C_{\hat{A}}^{\left( s-1\right) }\equiv%
\overline{R}^{(s-1)}\left( \mathcal{S}\right) +\vec{\pi}\cdot C^{\left(
s-1\right) }\func{mod}\ I,$ and since $\overline{R}_{\dag}^{(s-1)}\left( 
\mathcal{S}\right) \in M_{n}(\mathcal{O}_{E}[[\mathcal{S}]])^{\mid\tau\mid},$
the latter implies that $\overline{R}_{\hat{A}}^{\left( s-1\right) }\left( 
\mathcal{S}\right) \equiv\overline{R}^{\left( s-1\right) }\left( \mathcal{S}%
\right) \func{mod}\ I.$ Let $\Pi_{\dag}(\mathcal{S})=\Pi_{\dag}^{(0)}+\pi
\Pi_{\dag}^{(1)}+\pi^{2}\Pi_{\dag}^{(2)}+\cdots,$ and let $%
\Pi_{\dag}^{(0)}=\left(
P_{1,\dag},P_{2,\dag},\cdots,P_{f-1,\dag},P_{0,\dag}\right) .$ By the proof
of \cite[Lemma 4.1]{DO10b}, the matrices $H_{\dag}^{\left( s\right) }$ can
be chosen to be solutions of the equations 
\begin{equation}
H_{\dag}^{\left( s\right) }-\vec{p}^{(s-1)}\Pi_{\dag}^{(0)}(\mathcal{S}%
)\varphi\left( H_{\dag}^{\left( s\right) }\right) \left( \Pi_{\dag}^{(0)}(%
\mathcal{S})\right) ^{-1}=-\overline{R}_{\dag}^{\left( s-1\right) },\ 
\mathrm{with\ }\dag\in\{0,\hat{A}\}.  \label{equation 3*}
\end{equation}
\noindent Let $H_{\dag}^{(s)}=\left(
H_{1,\dag}^{(s)},H_{2,\dag}^{(s)},...,H_{f-1,\dag}^{(s)},H_{0,\dag}^{(s)}%
\right) $ and $-\overline {R}_{\dag}^{(s-1)}=\left( \overline{R}%
_{1,\dag}^{(s-1)},\overline{R}_{2,\dag }^{(s-1)},...,\overline{R}%
_{f-1,\dag}^{(s-1)},\overline{R}_{0,\dag}^{(s-1)}\right) .$

\noindent Equations (\ref{equation 3*}) are equivalent to the systems of
equations%
\begin{equation}
H_{i,\dag}^{\left( s\right) }-P_{i,\dag}\cdot H_{i+1,\dag}^{\left( s\right)
}\cdot\left( p^{s-1}P_{i,\dag}^{-1}\right) =\overline{R}_{i,\dag }^{\left(
s-1\right) },  \label{equation4*}
\end{equation}
for $i=1,2,...,f,$ $\dag\in\{0,\hat{A}\},$ \noindent and with indices viewed 
$\func{mod}$ $f.$ \noindent\noindent These imply that%
\begin{align}
H_{1,\dag}^{(s)}-Q_{f,\dag}H_{1,\dag}^{(s)}(p^{f(s-1)}Q_{f,\dag}^{-1})=%
\overline{R}_{1,\dag}^{(s-1)}+Q_{1,\dag}\overline{R}_{2,\dag}^{\left(
s-1\right) }(p^{(s-1)}Q_{1,\dag}^{-1})+Q_{2,\dag}\overline{R}_{3,\dag
}^{\left( s-1\right) }(p^{2(s-1)}Q_{2,\dag}^{-1}) &  \notag \\
+\cdots+Q_{f-1,\dag}\overline{R}_{0,\dag}^{(s-1)}(p^{(s-1)(f-1)}Q_{f-1,\dag
}^{-1}), &  \label{equation55*}
\end{align}
where$\ Q_{i,\dag}=P_{1,\dag}\cdots P_{i,\dag}\ $for$\mathrm{\ }$all$\
i=1,2,...,f.$ The matrices $H_{i,\dag}^{(s)},$ $i=2,3,...,f~$are uniquely
determined by the matrix $H_{1,\dag}^{(s)}.$ Let $V_{\dag}^{\left( s\right)
}=\overline{R}_{1,\dag}^{(s-1)}+Q_{1,\dag}\overline{R}_{2,\dag
}^{(s-1)}(p^{(s-1)}Q_{1,\dag}^{-1})+Q_{2,\dag}\overline{R}%
_{3,\dag}^{(s-1)}(p^{2(s-1)}Q_{2,\dag}^{-1})+\cdots+Q_{f-1,\dag}\overline{R}%
_{0,\dag }^{(s-1)}(p^{(s-1)(f-1)}Q_{f-1,\dag}^{-1}).$\ Since $\overline{R}%
_{i}^{(s-1)}\equiv\overline{R}_{i,\hat{A}}^{(s-1)}\func{mod}\ I,$ and since $%
Q_{i}\equiv Q_{i,\hat{A}}\func{mod}\ I$ and $p^{i(s-1)}Q_{i,\hat{A}%
}^{-1}\equiv p^{i(s-1)}Q_{i}^{-1}\func{mod}\ I$ for all $i,\ $it follows
that $V^{\left( s\right) }\equiv V_{\hat{A}}^{\left( s\right) }\func{mod}\
I. $ Since $Q_{1,\hat{A}}\equiv Q_{1}\func{mod}\ I,$ the latter and
equations (\ref{equation55*}) imply that $H_{1,\hat{A}%
}^{(s)}-H_{1}^{(s)}=Q_{f}\left( H_{1,\hat{A}}^{(s)}-H_{1}^{(s)}\right)
\left( p^{f(s-1)}Q_{f}^{-1}\right) $

\noindent$\func{mod}\ I,$ and Lemma \ref{daggermodp lemma} applied for $%
B=H_{1,\hat{A}}^{(s)}-H_{1}^{(s)}$ implies that $H_{1,\hat{A}}^{(s)}\equiv
H_{1}^{(s)}\func{mod}\ I.$ Since $P_{i,\hat{A}}\equiv P_{i}\func{mod}\ I$
and $R_{i,\hat{A}}^{\left( s-1\right) }\equiv R_{i}^{\left( s-1\right) }%
\func{mod}\ I,$ equations (\ref{equation4*}) imply that $H_{\hat{A}%
}^{(s)}\equiv H^{(s)}\func{mod}\ I.$ Since $G_{\gamma,\dag}^{(s)}=G_{\gamma,%
\dag}^{(s-1)}+\vec{\pi}^{s-1}H_{\dag}^{(s)},$ the inductive hypothesis
implies that $G_{\gamma,\hat{A}}^{(s)}(\mathcal{S})\equiv G_{\gamma}^{(s)}(%
\mathcal{S})\func{mod}\ I.$ Formula (\ref{einai swsto!}) yields $\vec{\pi}%
^{s}\cdot\gamma\Pi_{\hat{A}}\left( \mathcal{S}\right) \cdot$ $R_{\hat{A}%
}^{\left( s\right) }\equiv\vec{\pi}^{s}\cdot\gamma\Pi\left( \mathcal{S}%
\right) \cdot$ $R^{\left( s\right) }\func{mod}\ I.$ Since the coordinate
matrices of both the matrices $\gamma\Pi_{\hat{A}}\left( \mathcal{S}\right)
\ $and $\gamma\Pi\left( \mathcal{S}\right) $ coincide $\func{mod}\ I$ and
have nonzero determinants $\func{mod}\ I,$ it follows that $R_{\hat{A}%
}^{\left( s\right) }\equiv R^{\left( s\right) }\func{mod}\ I,$ and this
finishes the induction. We have shown that $G_{\gamma,\hat{A}}^{(s)}(%
\mathcal{S})\equiv G_{\gamma}^{(s)}(\mathcal{S})\func{mod}\ I$ for any $%
s\geq k.$ Since $G_{\gamma,\dag}(\mathcal{S})=\underset{s\rightarrow\infty}{%
\lim}G_{\gamma,\dag}^{\left( s\right) }(\mathcal{S})$ it follows that $%
G_{\gamma,\hat{A}}(\mathcal{S})\equiv G_{\gamma}(\mathcal{S})\func{mod}\ I,$
and this finishes the proof.
\end{proof}

\begin{corollary}
\label{meA=0}Let $A\in M_{2}\left( p^{1+\alpha \left( k-1\right) }\mathcal{O}%
_{E}[[\pi ]]\right) ^{\mid \tau \mid }.$ Then for any $\vec{a}\in \mathfrak{m%
}_{E}^{f}$ there exist $G_{K_{f}}$-stable $\mathcal{O}_{E}$-lattices with
respect to which $\overline{V}_{A}(\vec{a})\simeq \overline{V}_{0}(\vec{0}).$
\end{corollary}

\begin{proof}
Proposition \ref{dagger} implies that $\Pi _{\hat{A}}(\vec{a})\equiv \Pi (%
\vec{0})\func{mod}\ \mathfrak{m}_{E}$ and $G_{\gamma ,\hat{A}}(\vec{a}%
)\equiv G_{\gamma }(\vec{0})\func{mod}\ \mathfrak{m}_{E}$ for all $\gamma
\in \Gamma _{K}.$ The rest of the proof is identical to that of \cite[%
Theorem 4.7]{DO10b}.\noindent
\end{proof}

Parts (ii) and (iii) of Theorem A follow from Proposition \ref{xwris A=0}
and Corollary \ref{meA=0}. The following proposition proves part (i) and
finishes the proof of the theorem.

\begin{proposition}
\label{twmodptext}Let $A\in M_{2}\left( p^{\alpha \left( k-1\right) }%
\mathcal{O}_{E}[[\pi ]]\right) ^{\mid \tau \mid }$ and $\vec{\alpha}\in
\left( p^{m}\mathfrak{m}_{E}\right) ^{f}.$ We define the rank two filtered $%
\varphi $-module $\left( \mathbb{D}_{A}\left( \vec{\alpha}\right) ,\varphi
\right) $ with Frobenius endomorphism%
\begin{equation}
\left( \varphi \left( \eta _{1}\right) ,\varphi \left( \eta _{2}\right)
\right) :=\left( \eta _{1},\eta _{2}\right) P_{A}\left( \vec{\alpha}\right)
\label{fi}
\end{equation}%
and filtration%
\begin{equation*}
\ \ \ \ \ \mathrm{Fil}^{\mathrm{j}}(\mathbb{D}_{A}\left( \vec{\alpha}\right)
):=\left\{ 
\begin{array}{l}
E^{\mid \tau \mid }\eta _{1}\oplus E^{\mid \tau \mid }\eta _{2}\ \ \ \ \ \ 
\mathrm{if}\text{ }j\leq 0, \\ 
E^{\mid \tau _{I_{0}}\mid }\left( \vec{x}\eta _{1}+\vec{y}\eta _{2}\right) \
\ \ \mathrm{if}\text{ }1\leq j\leq w_{0}, \\ 
E^{\mid \tau _{I_{1}}\mid }\left( \vec{x}\eta _{1}+\vec{y}\eta _{2}\right) \
\ \ \mathrm{if}\text{ }1+w_{0}\leq j\leq w_{1}, \\ 
\ \ \ \ \ \ \ \ \ \ \ \ \ \ \ \ \ \ \ \ \cdots \cdots \\ 
E^{\mid \tau _{I_{t-1}}\mid }\left( \vec{x}\eta _{1}+\vec{y}\eta _{2}\right)
\ \mathrm{if}\text{ }1+w_{t-2}\leq j\leq w_{t-1}, \\ 
\ \ \ \ \ \ \ \ \ \ \ 0\ \ \ \ \ \ \ \ \ \ \ \ \ \mathrm{if}\text{ }j\geq
1+w_{t-1},%
\end{array}%
\right.
\end{equation*}%
with%
\begin{equation}
\ (x_{i},y_{i})=\left\{ 
\begin{array}{l}
(1,-\alpha _{i})\ \ \ \text{\ \ }\mathrm{if}\text{ }P_{i}\ \mathrm{has}\text{
}\mathrm{type}\text{ }1\ \mathrm{or}\ 2, \\ 
(-\alpha _{i},1)\ \ \ \ \ \mathrm{if}\text{ }P_{i}\ \mathrm{has}\text{ }%
\mathrm{type}\text{ }3\ \mathrm{or}\text{ }4.%
\end{array}%
\right.  \label{trallalo!}
\end{equation}%
The filtered $\varphi $-modules $\mathbb{D}_{A}\left( \vec{\alpha}\right) $
are admissible and correspond to $2$-dimensional crystalline $E$-linear
representations $V_{A}\left( \vec{\alpha}\right) $ of $G_{K_{f}}$ with
Hodge-Tate type $\mathrm{HT}\left( \tau _{i}\right) =\{0,-k_{i}\}.$
\end{proposition}

\begin{proof}
We compute the filtered $\varphi$-modules given rise to by the Wach modules $%
\mathbb{N}_{\hat{A}}(\vec{a}).$ The $\varphi$-action is obviously given by
formula (\ref{fi}), and it suffices to compute $\mathrm{Fil}^{\mathrm{j}%
}\left( \mathbb{N}_{\hat{A}}(\vec{a})/\pi\mathbb{N}_{\hat{A}}(\vec {a}%
)\right) .$ By Theorem \ref{berger thm}, $\vec{x}\eta_{1}+\vec{y}\eta _{2}\in%
\mathrm{Fil}^{\mathrm{j}}\left( \mathbb{N}_{\hat{A}}(\vec{a})\right) $ if
and only if there $\varphi\left( \vec{x}\eta_{1}+\vec{y}\eta_{2}\right) \in
q^{j}\mathbb{N}_{\hat{A}}(\vec{a}).$ Written in matrix form and recalling
that the $\varphi$-action on $\mathcal{O}_{E}[[\pi]]^{\mid\tau\mid}$ is
given by formula (\ref{actions1}), the latter is equivalent to the existence
of vectors $\vec{u}_{1},\ \vec{u}_{2}\in\mathcal{O}_{E}[[\pi]]^{\mid\tau%
\mid} $ such that%
\begin{equation}
\left( e_{i-1}\eta_{1},e_{i-1}\eta_{2}\right) \left( Id+\hat{A}_{i}\left( 
\vec{a}\right) \right) \Pi_{i}\left( a_{i}\right) \left( \varphi\left(
x_{i}\right) ,\varphi\left( y_{i}\right) \right) ^{t}=\left(
q^{j}e_{i-1}\eta_{1},q^{j}e_{i-1}\eta_{2}\right) \left(
u_{1}^{i-1},u_{2}^{i-1}\right) ^{t}  \label{filtrationtrick1}
\end{equation}
for all $i=0,1,...,f-1,\ $where $e_{i}\ $are the idempotents$\ $of $\mathcal{%
O}_{E}[[\pi]]^{\mid\tau\mid}.$ Let 
\begin{equation*}
\left( \zeta_{1}^{i-1},\zeta_{2}^{i-1}\right) :=\left( e_{i-1}\eta
_{1},e_{i-1}\eta_{2}\right) \left( Id+\hat{A}_{i}\left( \vec{a}\right)
\right) \ \text{and\ }\left( v_{1}^{i-1},v_{2}^{i-1}\right) ^{t}:=\left( Id+%
\hat{A}_{i}\left( \vec{a}\right) \right) ^{-1}\left(
u_{1}^{i-1},u_{2}^{i-1}\right) ^{t}.
\end{equation*}
Since $\hat{A}\equiv A\func{mod}\ \vec{\pi}\ $and $\det\left(
Id+A_{i}\right) \in\mathcal{O}_{E}^{\times}\ $it follows that $Id+\hat{A}%
_{i}\left( \vec{a}\right) \in\mathrm{GL}_{2}\left( \mathcal{O}%
_{E}[[\pi]]\right) $ for all $i.$ Therefore $\left(
\zeta_{1}^{i},\zeta_{2}^{i}\right) $ is an ordered basis of $e_{i}\mathbb{N}%
_{\hat{A}}(\vec{a})$ for all $i$ and equation (\ref{filtrationtrick1})
implies%
\begin{equation*}
\left( \zeta_{1}^{i-1},\zeta_{2}^{i-1}\right) \Pi_{i}\left( a_{i}\right)
\left( \varphi\left( x_{i}\right) ,\varphi\left( y_{i}\right) \right)
^{t}=\left( q^{j}\zeta_{1}^{i-1},q^{j}\zeta_{2}^{i-1}\right) \left(
v_{1}^{i-1},v_{2}^{i-1}\right) ^{t}.
\end{equation*}
Assume that $\Pi_{i}\left( X_{i}\right) $ is of type $2.$ Arguing as in the
proof of \cite[Proposition 5.22]{DO10b} we see that $x_{i},y_{i}\equiv0\func{%
mod}\ \pi$ if $j\geq k_{i}$ and $\pi^{j}\mid x_{i}+y_{i}a_{i}z_{i}$ for $%
1\leq j\leq k_{i}.$ Since $z_{i}\func{mod}\ \pi=p^{m}$ and $%
\alpha_{i}:=p^{m}a_{i},$ 
\begin{equation*}
\ \ \ \ \ \ e_{i}\vec{x}\eta_{1}+e_{i}\vec{y}\eta_{2}+\pi\mathbb{N}_{\hat{A}%
}(\vec{a})=\left\{ 
\begin{array}{c}
\alpha_{i}\bar{y}_{i}e_{i}\eta_{1}+\bar{y}_{i}e_{i}\eta_{2}+\pi\mathbb{N}_{%
\hat{A}}(\vec{a})\ \ \mathrm{if}\text{\ \ }1\leq\text{ }j\leq k_{i},\ \  \\ 
\ \ \ \ 0\ \ \ \ \ \ \ \ \ \ \ \ \ \ \ \ \ \ \ \ \ \ \ \ \mathrm{if}\text{\
\ }j\geq k_{i}%
\end{array}
\right.
\end{equation*}
where $\bar{y}_{i}=y_{i}\func{mod}\ \pi$ can be any element of $\mathcal{O}%
_{E}.$ Hence%
\begin{equation*}
e_{i}\mathrm{Fil}^{\mathrm{j}}(\mathbb{N}_{\hat{A}}(\vec{a})/\pi \mathbb{N}_{%
\hat{A}}(\vec{a}))=\left\{ 
\begin{array}{l}
e_{i}(\mathcal{O}_{E}^{\mid\tau\mid})\eta_{1}\tbigoplus e_{i}(\mathcal{O}%
_{E}^{\mid\tau\mid})\eta_{2}\ \text{if }j\leq0, \\ 
e_{i}(\mathcal{O}_{E}^{\mid\tau\mid})(\vec{x}\eta_{1}+\vec{y}\eta _{2})\ \ \
\ \ \ \text{if }1\leq j\leq k_{i}, \\ 
\ \ \ \ \ \ 0\ \ \ \ \ \ \ \ \ \ \ \ \ \ \ \ \ \ \ \ \ \ \ \ \text{if }%
j\geq1+k_{i},%
\end{array}
\right.
\end{equation*}
with$\ (x_{i},y_{i})=(-\alpha_{i},1).$ Computing for the other choices of $%
\Pi_{i}(a_{i}),$ we see that for all $i\in I_{0},$ $(x_{i},y_{i})$ is as in
formula (\ref{trallalo!}) and the proof follows as in \cite[Proposition 5.22]%
{DO10b}. To finish the proof, notice that by the definition of the
polynomials $z_{i}\ $appearing in the matrices $\Pi_{i},$ the sets $\{\left(
z_{1}a_{1}\func{mod}\ \pi,...,z_{0}a_{0}\func{mod}\ \pi\right) \},\ $where $%
\left( a_{1},...,a_{f}\right) \in\mathfrak{m}_{E}^{f}$ and $\left( p^{m}%
\mathfrak{m}_{E}\right) ^{f}\ $coincide.$\ $We let $\vec{\alpha}:=p^{m}\cdot%
\vec{a}$ for any vector $\vec{a}\in \mathfrak{m}_{E}^{f}$ and parametrize
our families by the vectors $\vec {\alpha}.$ Finally, since \noindent$%
\mathbb{D}_{A}\left( \vec{\alpha}\right) =\mathrm{Fil}^{\mathrm{j}}\mathbb{D}%
_{\mathrm{cris}}\left( V_{A}\left( \vec{\alpha}\right) \right) \cong
E^{\mid\tau\mid}\tbigotimes \nolimits_{\mathcal{O}_{E}{}^{\mid\tau\mid}}%
\mathrm{Fil}^{\mathrm{j}}\left( \mathbb{N}_{\hat{A}}(\vec{a})/\pi \mathbb{N}%
_{\hat{A}}(\vec{a})\right) ,\ $the filtered $\varphi$-modules $\mathbb{D}%
_{A}\left( \vec{\alpha}\right) $ are admissible because the $%
E^{\mid\tau\mid}\tbigotimes \nolimits_{\mathcal{O}_{E}{}^{\mid\tau\mid}}%
\left( \mathbb{N}_{\hat{A}}(\vec{a})/\pi\mathbb{N}_{\hat{A}}(\vec{a})\right) 
$ are admissible by Theorem \ref{berger thm}.
\end{proof}

\begin{remark}
\label{oneremark...}For fixed $\vec{a}$ the filtered $\varphi $-modules $%
\mathbb{N}_{\hat{A}}(\vec{a})/\pi \mathbb{N}_{\hat{A}}(\vec{a})$ only depend
on $A\equiv \hat{A}\func{mod}\ \pi .$ This is clear from the proof of
Proposition \ref{twmodptext} (essentially by Theorem \ref{berger thm}).
\end{remark}

\begin{acknowledgement}
The paper was written at Fields Institute during the 2012 Thematic Program
on Galois Representations. We thank the Fields Institute for its hospitality
and financial support. We also thank the anomymous referee for a very
careful reading of the paper.
\end{acknowledgement}

\noindent


\begin{thebibliography}{10}
\bibitem[1]{Ber03} Berger, L., Limites de repr\'{e}sentations cristallines,
Comp. Math.\textbf{\ }140, (2004)\textbf{\ }1473-1498.

\bibitem[2]{Ber12} Berger, L., Local constancy for the reduction $\func{mod}$
$p$ of $2$-dimensional crystalline representations. Bull. London Math. Soc.
Advance Access published November 7, 2011.

\bibitem[3]{BB04} Berger, L., Breuil, C., Towards a $p$-adic Langlands
programme. Summer school on $p$-adic arithmetic geometry in Hangzhou. %
\url{http://perso.ens-lyon.fr/laurent.berger/autrestextes/hangzhou.pdf}

\bibitem[4]{BLZ04} Berger, L., Li, H., Zhu, H.J., Construction of some
families of 2-dimensional crystalline representations. Math. Ann. 329,
(2004) 365-377.

\bibitem[5]{BDJ} Buzzard, K., Diamond, F., Jarvis, F., On Serre's conjecture
for mod $l$ Galois representations over totally real fields. Duke Math. J.
155 (2010), no. 1, 105--161.

\bibitem[6]{CF00} Colmez, P., Fontaine, J-M., Construction des repr\'{e}%
sentations p-adiques semi-stables. Invent. Math. 140, (2000) 1-43.

\bibitem[7]{DO10a} Dousmanis, G., Rank two filtered $(\varphi ,N)$-modules
with Galois descent data and coefficients. Trans. Amer. Math. Soc. 362
(2010), no. 7, 3883-3910.

\bibitem[8]{DO10b} Dousmanis, G., On reductions of families of crystalline
Galois representations. Doc. Math. 15 (2010), 873--938.

\bibitem[9]{FO88} Fontaine, J.-M., Le corpes des p\'{e}riodes $p$-adiques. P%
\'{e}riodes $p$-adiques (Bures-sur-Yvette, 1988). Asterisque 223, (1994)
59-111.

\bibitem[10]{FO90} Fontaine, J.-M., Repr\'{e}sentations $p$-adiques des
corps locaux I. The Grothendieck Festschrift, Vol II, (1990) 249-309, Prog.
Math. 87, Birkh\"{a}user Boston, Boston, MA.

\bibitem[11]{GLS12} Gee, T., Liu, T., Savitt, D., The Buzzard-Diamond-Jarvis
conjecture for unitary groups. \url{http://arxiv.org/abs/1203.2552}

\bibitem[12]{LI10} Liu, T., A note on lattices in semi-stable
representations. Mathematische Annalen 346 (2010), no. 1, 117--138.
\end{thebibliography}
\end{document}